\documentclass[a4paper]{article}
\usepackage{amsmath,amssymb,amsthm}
\usepackage{nccmath} 
\usepackage{mathrsfs} 
\usepackage{tikz} 
\usetikzlibrary{cd} 
\usetikzlibrary{arrows} 
\usepackage{comment} 
\usepackage{here} 
\newtheorem{prop}{Proposition}[section]
\newtheorem{thm}[prop]{Theorem}
\newtheorem{lem}[prop]{Lemma}
\newtheorem{cor}[prop]{Corollary}
\newtheorem{defn}[prop]{Definition}

\newtheorem{remark}{Remark}

\makeatletter

\@addtoreset{equation}{subsection}
\makeatother
\begin{document}
\title{A filtration on the higher Chow group of zero cycles on an abelian variety}
\author{Buntaro Kakinoki}
\date{}
\maketitle
\begin{abstract}
In this paper we extend Gazaki's results on the Chow groups of abelian varieties to the higher Chow groups.
We introduce a Gazaki type filtration on the higher Chow group of zero-cycles on an abelian variety, 
whose graded quotients are connected to the Somekawa type $K$-group. 
Via the \'{e}tale cycle map, we will compare this filtration 
with a filtration on the \'{e}tale cohomology induced by the Hochschild-Serre spectral sequence. 
As an application over local fields, 
we obtain an estimate of the kernel of the reciprocity map. 
\end{abstract}
%
\section{Introduction}
		Let $k$ be a field. Let $A$ be an abelian variety over $k$ of dimension $d$. 
In \cite{G}, Gazaki introduced a descending filtration $\{ F_{0}^{\nu} \}_{\nu \geq 0}$ 
on the Chow group $\mathrm{CH}_{0}(A) = \mathrm{CH}^{d}(A)$ of zero cycles on $A$
and calculated its graded quotients up to bounded torsion.
Let $s \geq 0$ be an integer.
Using the same method, we define a descending filtration $\{ F_{s}^{\nu} \}_{\nu \geq 0}$ on the higher Chow group $\mathrm{CH}^{d+s}(A, s)$ of zero cycles and calculate its graded quotients.
The structure of this paper is parallel with \cite{G}.

\begin{thm} 																						\label{Intro-1}
For any integers $r, s \geq 0$, there is a canonical isomorphism: 
$$
\Phi_{r,s}' : F_{s}^{r+s} / F_{s}^{r+s+1} \otimes \mathbb{Z}\left[ \frac{1}{r!} \right] 
\xrightarrow{\sim} S_{r}(k; A, \mathcal{K}_{s}^{\mathrm{M}}) \otimes \mathbb{Z}\left[ \frac{1}{r!} \right].
$$
Here $S_{r}(k; A, \mathcal{K}_{s}^{\mathrm{M}})$ is defined in section \ref{section1.1} as the quotient abelian group
of the Somekawa type $K$-group $K_{r}(k; A, \mathcal{K}_{s}^{\mathrm{M}})$ 
by an action of the $r$-th symmetric group $\mathfrak{S}_{r}$.
\end{thm}

This theorem is obtained by Gazaki in the case $s = 0$ \cite[Theorem 3.8]{G}. 
In this case, the first three steps of the filtration is described as 
$$
F_{0}^{0} = \mathrm{CH}_{0}(A) \supset 
F_{0}^{1} = \mathrm{Ker}(\deg_{A} : \mathrm{CH}_{0}(A) \twoheadrightarrow \mathbb{Z}) \supset 
F_{0}^{2} = \mathrm{Ker}(\mathrm{alb}_{A} : F_{0}^{1} \twoheadrightarrow A(k)).
$$
Moreover, the filtration $F_{0}^{*}$ agrees with the Bloch-Beauville filtration up to torsion \cite[Corollary 4.4, 4.5]{G}. 
We have $F_{0}^{d+1} \otimes \mathbb{Q} = 0$ 
by Bloch \cite[(0.1) Theorem]{Bl}, Beauville \cite[Proposition 1]{B} and Deninger-Murre \cite[Lemma 2.18]{DM}.
We refer the reader to \cite[Remark 4.6]{G} for a brief review of Beauville's argument.
These facts are generalized to $s \geq 0$ as follows.
Let $\pi : A \to \mathrm{Spec}\, k$ be the structure morphism.
The first few terms are given as follows:
$$
F_{s}^{0} = \dots = F_{s}^{s} = \mathrm{CH}^{d+s}(A, s) \supset 
F_{s}^{s+1} = \mathrm{Ker}( \pi_{*} : \mathrm{CH}^{d+s}(A, s) \twoheadrightarrow \mathrm{CH}^{s}(k, s) \cong K_{s}^{\mathrm{M}} (k) ).
$$
Up to torsion, each quotient agrees with the following eigenspace \cite[Proposition 4.8]{Su}: 
$$
F_{s}^{r+s} / F_{s}^{r+s+1} \otimes \mathbb{Q} \cong 
\{ \alpha \in \mathrm{CH}^{d+s}(A, s) \otimes \mathbb{Q} \, | \, 
m^{*} \alpha = m^{2d-r} \alpha \ \text{for all} \ m \in \mathbb{Z} \},
$$ 
considered in \cite{B} and \cite{Su}, where $m^{*}$ is the flat pull-back by the multiplication $m : A \to A$.
Sugiyama has shown in \cite[Theorem 1.3]{Su} that this eigenspace vanishes for $s \geq 0$, $r \geq 2d+1$.
One can ask whether $F_{s}^{\nu} = 0$ holds integrally for a sufficiently large $\nu$. 

Let $n$ be an integer invertible in $k$.
Denote the \'{e}tale cycle map by 
$$
\rho_{A, n}^{s} : \mathrm{CH}^{d+s}(A, s) \to H_{\text{\'{e}t}}^{2d+s}(A, \mathbb{Z}/n (d+s)).
$$
There is a descending filtration 
$H^{2d+s}(A, \mathbb{Z}/n (d+s)) = {\rm fil}_{{\rm HS}}^{0}H^{2d+s} \supset {\rm fil}_{{\rm HS}}^{1}H^{2d+s} \supset \dots \supset {\rm fil}_{{\rm HS}}^{2d+s}H^{2d+s} \supset 0$ 
induced by the Hochschild-Serre spectral sequence \eqref{H-S}.
In Proposition \ref{degeneration}, we give a sufficient condition 
for the spectral sequence \eqref{H-S} to degenerate at $E_{2}$-page.
Via the \'{e}tale cycle map, we compare the Gazaki type filtration $\{ F_{s}^{\nu} \}_{\nu \geq 0}$ 
with the filtration $\{ {\rm fil}_{{\rm HS}}^{\nu}H^{2d+s}\}_{\nu \geq 0}$.

\begin{thm}																						\label{Intro-2}
Let $r,s \geq 0$ be two integers.
Assume that $k$ is perfect, 
the Hochschild-Serre spectral sequence \eqref{H-S} degenerates at $E_{2}$-page 
and $n$ is coprime to $(r-1)!$.
Then we have
$$
\rho_{A,n}^{s} (F_{s}^{r+s}) \subset {\rm fil}_{{\rm HS}}^{r+s}H^{2d+s}.
$$
\end{thm}

From now on, let $k$ be a finite extension of $\mathbb{Q}_{p}$. 
We state further consequences of Theorem \ref{Intro-1}. 

Using a result of Raskind and Spiess \cite[Theorem 4.5, Remark 4.4.5]{RS}, 
Gazaki has shown in \cite[Corollary 6.3]{G} that 
if $A$ has split semi-ordinary reduction, 
$F_{0}^{2} / F_{0}^{3} \otimes \mathbb{Z}[1/2]$ is the direct sum of a finite group and a divisible group
and $F_{0}^{\nu} / F_{0}^{\nu + 1}$ is divisible if $\nu \geq 3$.
In particular, for $\nu \geq 3$, we obtain a decomposition
$$
\mathrm{Ker} (\mathrm{alb}_{A}) / F_{0}^{\nu} 
\cong F_{0}^{2} / F_{0}^{3} \oplus \dots \oplus F_{0}^{\nu - 1} / F_{0}^{\nu}.
$$

In Theorem \ref{div&kernel} (i), we obtain an analogue in the higher case 
by using a result of Yamazaki \cite[Lemma 2.4, Proposition 3.1]{Y2}. 
If $A$ has potentially good reduction or split semi-abelian reduction and $s > 0$, 
the quotient $F_{s}^{\nu} / F_{s}^{\nu + 1}$ is divisible
if $\nu \geq 3$. In particular, for $s > 0$ and $\nu > s+1$, we obtain a decomposition
$$
\mathrm{CH}^{d+s}(A, s) / F_{s}^{\nu} 
\cong K_{s}^{\mathrm{M}} (k) \oplus \mathrm{Ker} (\pi_{*}) 
\cong K_{s}^{\mathrm{M}} (k) \oplus F_{s}^{s+1} / F_{s}^{s+2} \oplus \dots \oplus F_{s}^{\nu -1} / F_{s}^{\nu}.
$$

As an application of Theorem \ref{Intro-1} and \ref{Intro-2}, 
Gazaki gave an estimate of the kernel of a homomorphism induced by 
the Brauer-Manin pairing by using the filtration $\{ F_{0}^{\nu} \}_{\nu \geq 0}$.
By using $\{ F_{1}^{\nu} \}_{\nu \geq 0}$ similarly, 
we give an estimate of the kernel of the reciprocity map 
$\mathrm{rec}_{A} : SK_{1}(A) \to \pi_{1}^{{\rm ab}}(A)$ (see \cite{Sa2}).
For an abelian group $B$, we denote by $B_{{\rm div}}$ the maximal divisible subgroup of $B$.

\begin{thm}
Let $k / \mathbb{Q}_{p}$ be a finite extension. 
Let $A$ be an abelian variety of dimension $d$.
Denote $K_{1} = \mathrm{Ker} (\mathrm{CH}^{d+1}(A, 1) \cong SK_{1}(A) \xrightarrow{\mathrm{rec}_{A}} \pi_{1}^{{\rm ab}}(A))$.
Then $F_{1}^{3} \subset K_{1} \subset F_{1}^{2}$.
Moreover if $A = \mathrm{Jac} (C)$ is the Jacobian variety with potentially good reduction or split semi-abelian reduction
of a smooth proper geometrically connected curve $C$ over $k$ with $C(k) \neq \emptyset$,
then $F_{1}^{2} / F_{1}^{3}$ is the direct sum of a finite group and a divisible group
and $K_{1} / F_{1}^{\nu} \otimes \mathbb{Z}[1/2] 
= ( F_{1}^{2} / F_{1}^{\nu} )_{{\rm div}} \otimes \mathbb{Z}[1/2]$ for any $\nu \geq 3$.
\end{thm}

\noindent \textit{Notation and Conventions.}
Throughout this paper, we fix a base field $k$.
For a scheme $X$ over $k$, let $X_{(0)}$ be the set of all closed points in $X$.
For $x \in X$, we denote the residue field by $k(x)$.
For an extension $F/k$ of fields, 
we denote a scheme $X \times_{k} F$ by $X_{F}$ and the set of all $F$-valued points by $X(F)$.
If $K$ is a function field in one variable over $k$ and $v$ is a place of $K/k$, 
then $\mathcal{O}_{v}$ and $k(v)$ denote the valuation ring and the residue field.
\section{Review}
In section \ref{section1.1}, we prepare a Somekawa type $K$-group $S_{r}(k; A, \mathcal{K}_{s}^{\mathrm{M}})$, 
which plays a key role in section \ref{section Gazaki type} 
to define the Gazaki type filtration on the higher Chow group of zero cycles for an abelian variety.
In section \ref{section higher Chow}, we review the cubical definition of the higher Chow group
and the Weil reciprocity for the higher Chow groups shown by Akhtar.
This is a key lemma for Theorem \ref{Psi_{r,s}}. 
In section \ref{section K-Y}, we review a theorem of Kahn and Yamazaki 
that will be used in Corollary \ref{Cor Somekawa type map}.
	\subsection{The group $S_{r}(k; A, \mathcal{K}_{s}^{\mathrm{M}})$} 							\label{section1.1}
		Let $G$ be a commutative group scheme over $k$.
For a sequence $E/F/k$ of fields, we have the restriction map $\mathrm{R}_{E/F} : G(F) \to G(E)$. 
If $E/F$ is finite, then we also have the trace map $\mathrm{Tr}_{E/F} : G(E) \to G(F)$, 
which satisfies
\begin{equation}																					\label{TrRes}
\mathrm{Tr}_{E/F} \circ \mathrm{R}_{E/F} = [E : F].
\end{equation}
The above is similar for the Milnor $K$-theory $K_{*}^{\mathrm{M}}$ of fields.

Let $A$ be an abelian variety over $k$.
For two integers $r,s\geq 0$, set
\begin{equation}																					\label{Tproduct}
T_{r}(k; A, \mathcal{K}_{s}^{\mathrm{M}}) = 
\bigoplus_{F/k ~ \textrm{finite}} A(F)^{\otimes r} \otimes K_{s}^{\mathrm{M}}(F), 
\end{equation}
where $F$ runs through all finite extensions of $k$, and $K_{s}^{\mathrm{M}}(F)$ is the Milnor $K$-theory of $F$.
We denote by 
$$
K_{r}(k; A, \mathcal{K}_{s}^{\mathrm{M}}) = K(k; \underbrace{A, \dots, A}_{r}, \mathcal{K}_{s}^{\mathrm{M}})
$$
the Somekawa type $K$-group which has been studied since \cite{S}.
It is the abelian group
$$
T_{r}(k; A, \mathcal{K}_{s}^{\mathrm{M}}) \big/ R,
$$
where $R$ is the subgroup generated by the following elements \eqref{rel-1}, \eqref{rel-2}.
We abbreviate $T_{r}(k; A, \mathcal{K}_{0}^{\mathrm{M}})$ and $K_{r}(k; A, \mathcal{K}_{0}^{\mathrm{M}})$ 
to $T_{r}(k; A)$ and $K_{r}(k; A)$. 
One has $K_{0}(k; A, \mathcal{K}_{s}^{\mathrm{M}}) \cong K_{s}^{\mathrm{M}}(k)$. 

Set $H_{i} = A$ for $i = 1, \dots, r$ and $H_{r+1} = \mathcal{K}_{s}^{\mathrm{M}}$. 
If $E/F/k$ is a sequence of finite field extensions and we have $h_{i_{0}} \in H_{i_{0}}(E)$ for some $i_{0} \in \{1, \dots, r+1\}$, 
and $h_{i} \in H_{i}(F)$ for all $i \neq i_{0}$, then
\begin{equation} 																							\label{rel-1}
h_{1} \otimes \dots \otimes \mathrm{Tr}_{E/F}(h_{i_{0}}) \otimes \dots \otimes h_{r+1} - 
\mathrm{R}_{E/F}(h_{1}) \otimes \dots \otimes h_{i_{0}} \otimes \dots \otimes \mathrm{R}_{E/F}(h_{r+1}) \in R.
\end{equation}

Let $K$ be a function field in one variable over $k$. 
Let $f_{1}, \dots, f_{r} \in A(K)$ and $g \in K_{s+1}^{\mathrm{M}}(K)$ . 
Then
\begin{equation} 																							\label{rel-2}
\sum_{v} s_{v}(f_{1}) \otimes \dots \otimes s_{v}(f_{r}) \otimes \partial_{v}(g) \in R
\end{equation}
where $v$ runs over all places of $K/k$. 
Here $\partial_{v} : K_{s+1}^{\mathrm{M}}(K) \to K_{s}^{\mathrm{M}}(k(v))$
is the boundery map in the Milnor $K$-theory, 
and $s_{v} : A(K) \to A(k(v))$ is the specialization map defined as the composition 
$A(K) \stackrel{\simeq}{\longleftarrow} A(\mathcal{O}_{v}) \to A(k(v))$ by the properness. 

For a finite field extension $F/k$ and $a_{1}, \dots, a_{r} \in A(F)$, $b \in K_{s}^{\mathrm{M}}(F)$ 
we denote by a symbol $\{a_{1}, \dots, a_{r}, b \}_{F/k}$ the class of 
$a_{1} \otimes \dots \otimes a_{r} \otimes b$ in $K_{r}(k; A, \mathcal{K}_{s}^{\mathrm{M}})$.

Let $\mathfrak{S}_{r}$ be the $r$-th symmetric group. 
An element $\sigma \in \mathfrak{S}_{r}$ acts on a symbol by 
$\sigma \cdot \{a_{1}, \dots, a_{r}, b \}_{F/k} = \{a_{\sigma(1)}, \dots, a_{\sigma(r)}, b \}_{F/k}$. 
This extends linearly to $K_{r}(k; A, \mathcal{K}_{s}^{\mathrm{M}})$. 
We denote the quotient abelian group of the action by
\begin{equation}																				\label{S type 1}
S_{r}(k; A, \mathcal{K}_{s}^{\mathrm{M}}) = K_{r}(k; A, \mathcal{K}_{s}^{\mathrm{M}}) \big/ \mathfrak{S}_{r}.
\end{equation}

Let $k'/k$ be a finite extension of fields.
The trace map of Somekawa type $K$-groups 
\begin{equation}																				\label{S trace}
\mathrm{Tr}_{k'/k} : K_{r}(k'; A_{k'}, \mathcal{K}_{s}^{\mathrm{M}}) \to K_{r}(k; A, \mathcal{K}_{s}^{\mathrm{M}})
\end{equation}
is defined as follows:
For a finite extension $F/k'$ and $a_{1}, \dots, a_{r} \in A_{k'}(F)$, $b \in K_{s}^{\mathrm{M}}(F)$, 
the map $\mathrm{Tr}_{k'/k}$ sends a symbol $\{a_{1}, \dots, a_{r}, b \}_{F/k'}$ 
to a symbol $\{a_{1}, \dots, a_{r}, b \}_{F/k}$.
This induces 
$\mathrm{Tr}_{k'/k} : S_{r}(k'; A_{k'}, \mathcal{K}_{s}^{\mathrm{M}}) \to S_{r,s}(k; A, \mathcal{K}_{s}^{\mathrm{M}})$.

We review an easy result on the Milnor $K$-theory.
Let $K$ be a function field in one variable over $k$. 
Let $g_{1}, \dots, g_{s+1} \in K^{*}$ be such that for every place $v$ of $K/k$ there exists $i(v) \in \{ 1, \dots, s+1 \}$ 
such that $g_{i} \in \mathcal{O}_{v}^{*}$ for all $i \neq i(v)$. 
Then for every place $v$ of $K/k$
\begin{equation}																	\label{boundary-order}
\partial_{v}(\{ g_{1}, \dots, g_{s+1} \}) 
= (-1)^{i(v)-1}\mathrm{ord}_{v} (g_{i(v)})\{ g_{1}(v), \dots, \Hat{g}_{i(v)}, \dots, g_{s+1}(v)  \} \in K_{s}^{\mathrm{M}}(k(v)),
\end{equation}
where $\Hat{g}_{i(v)}$ excludes the $i(v)$-th component and 
$v = \mathrm{ord}_{v} : K^{*} \twoheadrightarrow \mathbb{Z}$ is the normalized discrete valuation.

Let $g_{1}, \dots, g_{s}, h \in K^{*}$ be such that for every place $v$ of $K/k$ there exists $i(v) \in \{ 1, \dots, s \}$ 
such that $g_{i} \in \mathcal{O}_{v}^{*}$ for all $i \neq i(v)$. When $s=0$ this simply means we have $h \in K^{*}$. 
Then for every place $v$ of $K/k$
\begin{equation}																	\label{boundary-tame symbol}
\partial_{v}(\{ h, g_{1}, \dots, g_{s} \}) 
= \begin{cases} \displaystyle
\mathrm{ord}_{v} (h) & \text{if} \ s=0 \\ \displaystyle
\{ g_{1}(v), \dots, T_{v}(g_{i(v)}, h), \dots, g_{s}(v)  \} \in K_{s}^{\mathrm{M}}(k(v)) & \text{if} \ s>0
\end{cases}
\end{equation}
Here $T_{v} : K^{*} \times K^{*} \to k(v)^{*}$ is the tame symbol defined by 
$T_{v}(g,h) = (-1)^{v(g)v(h)} ( g^{v(h)} / h^{v(g)} ) (v)$ for $g,h \in K^{*}$.
	\subsection{Higher Chow groups}														\label{section higher Chow}
Let $X$ be an equi-dimensional scheme of finite type over $k$. 
In \cite{T}, one defines the $n$-cube $\square_{k}^{n}$ to be $(\mathbb{P}_{k}^{1} - \{ 1 \})^{n}$
and introduces the faces of $\square^{n}$, the cubical complex $c^{m}(X, \bullet)$ for an integer $m \geq 0$,
and the subcomplex $c^{m}(X, \bullet)_{\rm{degn}} \subset c^{m}(X, \bullet)$ of degenerate cycles. 
We denote the quotient complex $c^{m}(X, \bullet) / c^{m}(X, \bullet)_{\rm{degn}}$ by $z^{m}(X, \bullet)$, 
and define the \textit{higher Chow group} $\mathrm{CH}^{m}(X, n)$ 
as the $n$-th homology group of the complex $z^{m}(X, \bullet)$.

We focus on the higher Chow group of zero cycles. 
Let $d$ be the dimension of $X$ and $s \geq 0$ an integer. 
Then we obtain 
\begin{equation}																					\label{CH0}
\mathrm{CH}^{d+s}(X, s) = c^{d+s}(X, s) / d_{s+1}c^{d+s}(X, s+1),
\end{equation}
where $d_{s+1} : c^{d+s}(X, s+1) \to c^{d+s}(X, s)$ is the boundary map.
The object $c^{d+s}(X, s)$ is the free abelian group 
generated by all closed points in $X \times (\square^{1} - \{0,\infty \})^{s}$, 
and $c^{d+s}(X, s+1)$ is the free abelian group generated by all integral curves in $X \times \square^{s+1}$ 
which meets the codimension-$1$ faces in finitely many points and which does not meet the codimension-$2$ faces.
Let $C \in c^{d+s}(X, s+1)$ be an integral curve with function field $K = k(C)$. 
Let $p_{i} : X \times \square^{s+1} \to \square^{1}$ be the projection to the $i$-th cube and 
let $g_{i} : C \to \square^{1}$ be the composition 
$C \stackrel{\iota_{C}}{\hookrightarrow} X \times \square^{s+1} \xrightarrow{p_{i}} \square^{1}$. 
Let $\varphi : \Tilde{C} \to C$ be the normalization of $C$ and let $\Tilde{g}_{i}$ be the composition $g_{i} \circ \varphi$. 
We use $(t_{1}, \dots, t_{s+1})$ for the affine coordinate of $\square^{s+1}$ around $(0, \dots, 0)$.
For $\epsilon \in \{ 0, \infty \}$, we denote by $\varphi_{i}^{\epsilon}$ and $\iota_{C, i}^{\epsilon}$ 
the base change of $\varphi$ and $\iota_{C}$ by $t_{i} = \epsilon$ in the following cartesian diagram:
\begin{equation}                                                                                                                                 \label{diagram}
\begin{tikzcd}
\Tilde{g}_{i}^{-1}(\epsilon) \arrow[r, "\varphi_{i}^{\epsilon}"]\arrow[d, hook] &
g_{i}^{-1}(\epsilon) \arrow[r, hook, "\iota_{C, i}^{\epsilon}"]\arrow[d, hook] &
X \times \square^{s} \arrow[r]\arrow[d, hook, "t_{i} = \epsilon"] &
\mathrm{Spec} k \arrow[d, hook, "t = \epsilon"] \\
\Tilde{C} \arrow[r, "\varphi"'] &
C \arrow[r, hook, "\iota_{C}"'] &
X \times \square^{s+1} \arrow[r, "p_{i}"'] &
\square^{1}.
\end{tikzcd}
\end{equation}
Set $\sigma_{i}^{\epsilon} = \iota_{C, i}^{\epsilon} \circ \varphi_{i}^{\epsilon}$.
Given a closed point $w \in \coprod_{i,\epsilon} \Tilde{g}_{i}^{-1}(\epsilon) \subset \Tilde{C}$, 
there is a unique pair $(i(w), \epsilon(w))$ such that $\varphi(w) \in g_{i(w)}^{-1}(\epsilon(w))$
since all cycles $[g_{i}^{-1}(\epsilon)]$ have disjoint suppots on $X\times \square^{s}$ by the face condition. 
The boundary map $d_{s+1}$ is described as follows:
\begin{equation}                                                                                                                       \label{dC-2}
d_{s+1}(C) = 
\sum_{w \in \Tilde{C}} (-1)^{i(w)-1}
\mathrm{ord}_{w}(\Tilde{g}_{i(w)}) [k(w):k(\varphi(w))] \cdot \sigma_{i(w)}^{\epsilon(w)}(w),
\end{equation}
where we consider as $\Tilde{g}_{i} \in K^{*}$.

		We use the following isomorphism repeatedly.
\begin{thm} {\rm(Nesterenko-Suslin, Totaro \cite[Theorem 1]{T})}
Let $s \geq 0$ be an integer. There is a canonical isomorphism
\begin{equation}																						\label{N-S,T}
[\ ]_{k} : K_{s}^{\mathrm{M}}(k) \xrightarrow{\sim} \mathrm{CH}^{s}(k,s).
\end{equation}
\end{thm}

We review the Weil reciprocity for the higher Chow groups. 
Refer to \cite{A} for more general statement.
%
%
\begin{thm} \label{1.4.} {\rm(\cite{A}Theorem 4.5, Lemma 5.4, 6.3, 6.4, 6.5)}
Let $s \geq 0$ be an integer. 
Let $C$ be a proper smooth curve over $k$ with function field $K = k(C)$. 
For a closed point $P \in X_{K}$ and $g \in K_{s+1}^{\mathrm{M}}(K)$,  
it holds that
$$
\sum_{v \in C_{(0)}} \mathrm{Tr}_{k(v)/k} \left( 
s_{v}([P]) \times [\partial_{v} (g)]_{k(v)} \right) = 0 
\in \mathrm{CH}^{d+s}(X, s),
$$
where $s_{v} : \mathrm{CH}_{0}(X_{K}) \to \mathrm{CH}_{0}(X_{k(v)})$ 
is the specialization map for Chow groups in {\rm \cite[Section 20.3]{F}}, 
and $\times$ is the exterior product of higher Chow groups, 
and $\mathrm{Tr}_{k(v)/k} : \mathrm{CH}^{d+s}(X_{k(v)}, s) \to \mathrm{CH}^{d+s}(X, s)$ is the proper push-forward.
\end{thm}
	\subsection{A result of Kahn and Yamazaki}									\label{section K-Y}
		Let $A$ be an abelian variety over $k$. 
\begin{defn}																				\label{SomekawaK}
For $r, s \geq 0$, we define the group
$$
K_{r,s}(k; A, \mathbb{G}_{\mathrm{m}}) 
= \left[ \bigoplus_{F/k} A(F)^{\otimes r} \otimes (F^{*})^{\otimes s} \right] \Big/ R'
$$
where $F$ runs through all finite extensions of $k$ and $R'$ is the subgroup generated by the following elements \eqref{rel-3} and \eqref{rel-4}: 

Set $H_{i} = A$ for $i = 1, \dots, r$ and $H_{i} = \mathbb{G}_{\mathrm{m}}$ for $i = r+1, \dots, r+s$. 
If $E/F/k$ is a sequence of finite field extensions and we have $h_{i_{0}} \in H_{i_{0}}(E)$ for some $i_{0} \in \{1, \dots, r+s\}$, 
and $h_{i} \in H_{i}(F)$ for all $i \neq i_{0}$, then
\begin{equation} 																						\label{rel-3}
h_{1} \otimes \dots \otimes \mathrm{Tr}_{E/F}(h_{i_{0}}) \otimes \dots \otimes h_{r+s} - 
\mathrm{R}_{E/F}(h_{1}) \otimes \dots \otimes h_{i_{0}} \otimes \dots \otimes \mathrm{R}_{E/F}(h_{r+s}) \in R',
\end{equation}
where $\mathrm{R}_{E/F}$ is the restriction map and $\mathrm{Tr}_{E/F}$ is the trace map. 

Let $K$ be a function field in one variable over $k$. Let $f_{1}, \dots, f_{r} \in A(K)$ and $g_{1}, \dots g_{s}, h \in K^{*}$ 
such that for every place $v$ of $K/k$ there exists $i(v) \in \{ 1, \dots, s \}$ 
such that $g_{i} \in \mathcal{O}_{v}^{*}$ for all $i \neq i(v)$. Then
\begin{equation} 																						\label{rel-4}
\begin{aligned}
&\ \ \ \ \ \ \ \ \ \ \ \ \ \ \ \ \ \ \ \ \ \ 
\sum_{v} v(h) \cdot s_{v}(f_{1}) \otimes \dots \otimes s_{v}(f_{r}) \in R' \ &\text{if} \ s = 0,\\
&\sum_{v} s_{v}(f_{1}) \otimes \dots \otimes s_{v}(f_{r}) \otimes 
g_{1}(v) \otimes \dots \otimes T_{v}(g_{i(v)}, h) \otimes \dots \otimes g_{s}(v) \in R' \ &\text{if} \ s > 0,
\end{aligned}
\end{equation}
where $v$ runs over all places of $K/k$. 
Here $s_{v}$ and $T_{v}$ are from \eqref{rel-2} and \eqref{boundary-tame symbol}. 

As with \eqref{S type 1}, define 
\begin{equation}																						\label{S type 2}
S_{r,s}(k; A, \mathbb{G}_{\mathrm{m}}) = K_{r,s}(k; A, \mathbb{G}_{\mathrm{m}}) / \mathfrak{S}_{r}.
\end{equation}
\end{defn}

From a relation \eqref{boundary-tame symbol}, there exists a natural surjection
\begin{equation}																	\label{two somekawa}
K_{r,s}(k; A, \mathbb{G}_{\mathrm{m}}) \twoheadrightarrow K_{r}(k; A, \mathcal{K}_{s}^{{\rm M}}).
\end{equation}
When $s=0,1$, these two groups agree by definition.
When $r=0$, this is an isomorphism by \cite[Theorem 1.4]{S}.

\begin{thm} \label{KY} {\rm \cite[11.14. Theorem]{KY}}
If $k$ is perfect, the above morphism \eqref{two somekawa} is an isomorphism.
\end{thm}
Refer \cite{KY} for the proof. 
They used Voevodsky's triangulated category $\mathbf{DM}_{-}^{\text{eff}}$ of effective motivic complexes 
and constructed following horizontal isomorphisms:
\begin{center}
\begin{tikzcd}
K_{r,s}(k; A, \mathbb{G}_{\mathrm{m}}) \arrow[d, twoheadrightarrow]\arrow[r, "\sim"]
&\mathrm{Hom}_{\mathbf{DM}_{-}^{\text{eff}}} \left( \mathbb{Z}, \ 
A^{\otimes r} \otimes \mathbb{G}_{\mathrm{m}}^{\otimes s} \right) \arrow{d}[sloped,above]{\sim} \\
K_{r}(k; A, \mathcal{K}_{s}^{{\rm M}}) \arrow[r, "\sim"]
&\mathrm{Hom}_{\mathbf{DM}_{-}^{\text{eff}}} \left( \mathbb{Z}, \ 
A^{\otimes r} \otimes \mathcal{K}_{s}^{{\rm M}} \right)
\end{tikzcd}
\end{center}
where the tensor products on the right hand side are in the abelian category $\mathbf{HI}_{\mathrm{Nis}}$ 
of homotopy invariant Nisnevich sheaves with transfers. 
We have an isomorphism
$\mathbb{G}_{\mathrm{m}}^{\otimes s} \xrightarrow{\sim} \mathcal{K}_{s}^{{\rm M}}$ 
as a direct consequence of Suslin-Voevodsky's theorem \cite[Theorem 3.4]{SV} (see also \cite[1.3]{KY}), 
which induces the right vertical isomorphism. 
\section{Gazaki type filtration}													\label{section Gazaki type}
Throughout the rest of this paper, let $k$ be a field and $A$ an abelian variety over $k$ of dimension $d$.
	\subsection{The homomorphism $\Phi_{r,s}'$}
		Recall the description \eqref{CH0}. 
For a closed point $x$ in $A\times_{k} (\square_{k}^{1} - \{0, \infty \})^{s}$, 
we also denote the closed immersion by $x : \mathrm{Spec} \ k(x) \hookrightarrow A\times_{k} \mathbb{G}_{\mathrm{m}}^{s}$.
Denote the projection to $A$ and the $i$-th component by 
$p_{A} : A \times \mathbb{G}_{\mathrm{m}}^{s} \to A$ and 
$p_{i} : A \times \mathbb{G}_{\mathrm{m}}^{s} \to \mathbb{G}_{\mathrm{m}}$ for $1 \leq i \leq s$.

\begin{prop}																					\label{prop phi}
For any $r,s \geq 0$
the homomorphism
\begin{align*}
\phi_{r,s}^{k} : c^{d+s}(A, s) &\to K_{r}(k; A, \mathcal{K}_{s}^{\mathrm{M}}) \\
[x] &\mapsto \{ \underbrace{p_{A} \circ x, \dots, p_{A} \circ x}_{r \, \text{copies}}, 
\{ p_{1} \circ x, \dots, p_{s} \circ x \} \}_{k(x)/k},
\end{align*}
where $x$ is a closed point of $A\times_{k} (\square_{k}^{1} - \{0, \infty \})^{s}$, 
induces a map $\Phi_{r,s}^{k} : \mathrm{CH}^{d+s}(A, s) \to K_{r}(k; A, \mathcal{K}_{s}^{\mathrm{M}})$.
\end{prop}

When there is no confusion, we omit $k$. 
For a map composed with the natural projection to $S_{r}(k; A, \mathcal{K}_{s}^{\mathrm{M}})$, 
we use $\Phi_{r,s}' : \mathrm{CH}^{d+s}(A, s) \to S_{r}(k; A, \mathcal{K}_{s}^{\mathrm{M}})$.

\begin{proof}
Let $C \in c^{d+s}(A, s+1)$ be an integral curve with function field $K = k(C)$, 
let $\varphi : \Tilde{C} \to C$ be the normalization of $C$.
We denote by $q_{A} : A \times \square^{s+1} \to A$ the projection to $A$ 
and by $q_{i} : A \times \square^{s+1} \to \square^{1}$ the projection to the $i$-th cube for $1 \leq i \leq s+1$. 
Let $f : C \to A$ be the composition $C \hookrightarrow A \times \square^{s+1} \xrightarrow{q_{A}} A$ 
and let $g_{i} : C \to \square^{1}$ be the composition 
$C \hookrightarrow A \times \square^{s+1} \xrightarrow{q_{i}} \square^{1}$. 
Set $\Tilde{f} = f \circ \varphi$ and $\Tilde{g}_{i} = g_{i} \circ \varphi$.

Given a closed point $w \in \Tilde{C}$, 
there is a unique pair $(i(w), \epsilon(w))$ such that $\varphi(w) \in g_{i(w)}^{-1}(\epsilon(w))$
if and only if $\mathrm{ord}_{w}(g_{i}) \neq 0$ for some $i$ 
by the observation at \eqref{dC-2}.
If $\varphi(w) \in g_{i(w)}^{-1}(\epsilon(w))$, we obtain the cartesian diagram:
$$
\begin{tikzcd}
	\mathrm{Spec} \ k(\varphi(w)) 
	\arrow[from=1-1, to=1-3, bend left=20, hook, "\sigma_{i(w)}^{\epsilon(w)}(w)"]
	\arrow[r,hook]\arrow[d,equal]&
	g_{i(w)}^{-1}(\epsilon(w)) \arrow[r,hook]\arrow[d,hook]&
	A \times \square^{s} \arrow[r,"p_{1}\times \dots \times p_{s}"]\arrow[d,hook]&
	[2cm]\square^{s} \arrow[d, hook, "t_{i(w)} = \epsilon(w)"] \\
	\mathrm{Spec} \ k(\varphi(w)) \arrow[r,hook,"\varphi(w)"']&
	C \arrow[r,hook]&
	A \times \square^{s+1} \arrow[r, "q_{1}\times \dots \times q_{s+1}"']&
	\square^{s+1}
\end{tikzcd}
$$
where $\sigma_{i}^{\epsilon}$ is a map defined at \eqref{diagram}.
By the above diagram, we get 
$$
\phi_{r,s}(\sigma_{i(w)}^{\epsilon(w)}(w)) = 
\{f \circ \varphi(w), \dots, f \circ \varphi(w), 
\{ g_{1}(\varphi(w)), \dots, \Hat{g}_{i(w)}, \dots, g_{s+1}(\varphi(w)) \} \}_{k(\varphi(w))/k}
$$
where $\Hat{g}_{i(w)}$ means the exclusion of the $i(w)$-th component, 
and we consider as $f \circ \varphi(w) \in A(k(\varphi(w)))$ and 
$g_{i} \in \mathcal{O}_{\varphi(w)}^{*} \subset K$ for $i \neq i(w)$. 
By \eqref{dC-2}, we have 
\begin{fleqn}[0pt]
\begin{align*}
&\phi_{r,s} (d_{s+1}(C)) \\
&\scalebox{0.9}[1]{$\displaystyle
= \sum_{w \in \Tilde{C}} (-1)^{i(w)-1}\mathrm{ord}_{w}(\Tilde{g}_{i(w)}) 
\left\{ [k(w):k(\varphi(w))] f \circ \varphi(w), \dots, f \circ \varphi(w), 
\{ g_{1}(\varphi(w)), \dots, \Hat{g}_{i(w)}, \dots, g_{s+1}(\varphi(w))  \} \right\}_{k(\varphi(w))/k} 
$}
\end{align*}
\end{fleqn}
We write $\mathrm{R} = \mathrm{R}_{k(w)/k(\varphi(w))}$. 
By \eqref{TrRes}, \eqref{rel-1} and \eqref{boundary-order}, the above is equal to 
\begin{fleqn}[0pt]
\begin{align*}
&= \sum_{w \in \Tilde{C}} (-1)^{i(w)-1}\mathrm{ord}_{w}(\Tilde{g}_{i(w)}) 
\{ \mathrm{R}(f \circ \varphi(w)), \dots, \mathrm{R}(f \circ \varphi(w)), 
\{ \Tilde{g}_{1}(w), \dots, \Hat{\Tilde{g}}_{i(w)}, \dots, \Tilde{g}_{s+1}(w)  \} \}_{k(w)/k} \\
&= \sum_{w \in \Tilde{C}} 
\{ \mathrm{R}(f \circ \varphi(w)), \dots, \mathrm{R}(f \circ \varphi(w)), 
\partial_{w} \left\{ \Tilde{g}_{1}, \Tilde{g}_{2}, \dots, \Tilde{g}_{s+1}  \right\} \}_{k(w)/k}.
\end{align*}
\end{fleqn}

We have the following commutative diagram:
\begin{center}
\begin{tikzcd}
	&
	\mathrm{Spec} \ k(w) \arrow[dl, hook', "r_{w}"']\arrow[r]\arrow[d, hook, "w"'] &
	\mathrm{Spec} \ k(\varphi(w)) \arrow[d, "f \circ \varphi(w)"] \\
	\mathrm{Spec} \ \mathcal{O}_{w} \arrow[r, "i_{w}"'] &
	\Tilde{C} \arrow[r, "\Tilde{f}"'] &
	A
\end{tikzcd}
\end{center}
so that $\mathrm{R}_{k(w)/k(\varphi(w))} ( f \circ \varphi(w) ) = \Tilde{f} \circ i_{w} \circ r_{w}$.
Let $\eta : \mathrm{Spec} \ K \to \Tilde{C}$ be the generic point inclusion. 
The morphism $\Tilde{f} \circ i_{w}$ fits into the following commutative diagram:
\begin{center}
\begin{tikzcd}
	\mathrm{Spec} \ K \arrow[r, "\Tilde{f} \circ \eta"]\arrow[d] &
	A \arrow[d] \\
	\mathrm{Spec} \ \mathcal{O}_{w} \arrow[ru, "\Tilde{f} \circ i_{w}"]\arrow[r] &
	\mathrm{Spec} k
\end{tikzcd}
\end{center}
Therefore we obtain 
$\mathrm{R}_{k(w)/k(\varphi(w))} ( f \circ \varphi(w) ) = \Tilde{f} \circ i_{w} \circ r_{w} = s_{w}(\Tilde{f} \circ \eta)$ 
for every closed point $w$ in $\Tilde{C}$. 

Let $\mathbb{P}(\Tilde{C})$ be the smooth compactification of $\Tilde{C}$. 
Then for every $w \in \mathbb{P}(\Tilde{C})-\Tilde{C}$, 
there exists $i(w) \in \{1, \dots, s+1\}$ such that $\Tilde{g}_{i(w)}(w) = 1$ (\cite{A}Lemma 6.6). 
Therefore we obtain a relation in $K_{r} (k; A, \mathcal{K}_{s}^{{\rm M}}):$
\begin{align*}
\phi_{r,s} ( d_{s+1}(C) ) 
&= \sum_{w \in \mathbb{P}(\Tilde{C})} 
\{s_{w}(\Tilde{f} \circ \eta), \dots, s_{w}(\Tilde{f} \circ \eta), 
\partial_{w} \left\{ \Tilde{g}_{1}, \Tilde{g}_{2}, \dots, \Tilde{g}_{s+1}  \right\} \}_{k(w)/k} \\
&= 0.
\end{align*}
This concludes the proof.
\end{proof} 
		\begin{lem}                                                       \label{commutativity}
Let $k'/k$ be a finite extension of fields. 
Then the homomorphism $\Phi_{r,s}$ commutes with the push-forward $\mathrm{Tr}_{k'/k}$.
$$
\begin{tikzcd}
	\mathrm{CH}^{d+s}(A_{k'}, s) \arrow[r, "\Phi_{r,s}^{k'}"] \arrow[d, "\mathrm{Tr}_{k'/k}"'] 
	& K_{r}(k'; A_{k'}, \mathcal{K}_{s}^{\mathrm{M}}) \arrow[d, "\mathrm{Tr}_{k'/k}"] \\
	\mathrm{CH}^{d+s}(A, s) \arrow[r, "\Phi_{r,s}^{k}"'] & K_{r}(k; A, \mathcal{K}_{s}^{\mathrm{M}})
\end{tikzcd}
$$
where the left vertical map is the proper push-forward of higher Chow groups, 
and the right vertical map is the trace map of Somekawa type $K$-groups defined at \eqref{S trace}.
\end{lem}

The proof is straightforward. 
	\subsection{The homomorphism $\Psi_{r,s}'$}
		\noindent \textit{Notation and Observation.}
We have the degree map $\mathrm{deg}_{A} : \mathrm{CH}_{0}(A) \twoheadrightarrow \mathbb{Z}$.
Define $\mathrm{A}_{0}(A) := \mathrm{Ker(deg_{A})}$. 
%
(i) Let $0$ be the unit of $A$.
Let $F/k$ be a field extention. 
For $a \in A(F)$, we denote by $[a]_{F} \in \mathrm{CH}_{0}(A_{F})$ the class of $a$ 
and define $\lambda_{F}(a) = [a]_{F} - [0]_{F} \in \mathrm{A}_{0}(A_{F})$. 
%
\\
(ii) Let $m : A \times_{k} A \to A$ be the multiplication morphism on $A$.
We also denote $m$ by $+$. 
Recall that the Pontryagin product is defined by
\begin{equation}																		\label{Pontryagin}
\begin{aligned}
* : \mathrm{CH}_{0}(A) &\otimes \mathrm{CH}_{0}(A) \to \mathrm{CH}_{0}(A) \\
\alpha &\otimes \beta \mapsto m_{*}(\alpha \times \beta).
\end{aligned}
\end{equation}
This gives a ring structure on $\mathrm{CH}_{0}(A)$. 
If $x, y \in A(k)$, then we have $[x]_{k} * [y]_{k} = [x + y]_{k}$ by definition. 
For a field extension $F/k$ and $x, y \in A(F)$, we obtain
\begin{equation}                                                                                                                                \label{elementary-1}
\lambda_{F}(x + y) - \lambda_{F}(x) - \lambda_{F}(y) = \lambda_{F}(x) * \lambda_{F}(y) \in \mathrm{A}_{0}(A_{F}).
\end{equation}
If $p : A \to B$ is a homomorphism of abelian varieties over $k$ or the structure morphism of $A$, 
the proper push-forward $p_{*} : \mathrm{CH}_{0}(A) \to \mathrm{CH}_{0}(B)$ is a ring homomorphism, 
where the ring structure on $\mathrm{CH}_{0}(k)$ is the one compatible with $\mathbb{Z}$, 
which is defined by \eqref{Pontryagin} with 
the natural isomorphism $m : \mathrm{Spec} \, k \times_{k} \mathrm{Spec} \, k \to \mathrm{Spec} \, k$. 
In particular $\mathrm{deg}_{A}$ is a ring homomorphism. 
The subgroup $\mathrm{A}_{0}(A) = \mathrm{Ker}(\deg_{A})$ is an ideal of $\mathrm{CH}_{0}(A)$ 
with respect to the Pontryagin product. 
If $F/k$ is a finite field extension, 
the proper push-forward $\mathrm{Tr}_{F/k} : \mathrm{CH}_{0}(A_{F}) \to \mathrm{CH}_{0}(A)$ 
is also a ring homomorphism. 
		\begin{lem}																	\label{lemma psi}
For $r, s \geq 0$, we define the map
\begin{align*}
\psi_{r,s} : \coprod_{F/k} 
&\underbrace{A(F) \times \dots \times A(F)}_{r} \times K_{s}^{\mathrm{M}}(F)
\to \mathrm{CH}^{d+s}(A, s) \\
&(a_{1}, \dots, a_{r}, b)_{F} 
\mapsto \mathrm{Tr}_{F/k} \big(
\left( \lambda_{F}(a_{1}) * \dots * \lambda_{F}(a_{r}) \right)
\times [b]_{F} \big)
\end{align*}
where $F$ runs through all finite extensions of $k$, 
$\times$ is the exterior product, and $\mathrm{Tr}_{F/k}$ is the proper push-forward. 
See \eqref{N-S,T} and \textit{Notation and Observation} {\rm (i)} for $[\ ]_{F}$ and $\lambda_{F}$.
Then $\psi_{r,s}$ satisfies following properties: \\
{\rm(i)} For $a_{1}, \dots, a_{i,1}, a_{i,2}, \dots, a_{r} \in A(F)$, $b \in K_{s}^{\mathrm{M}}(F)$, 
\begin{equation}                                                                                                              \label{property-1}
\begin{aligned}
&\psi_{r,s}(a_{1}, \dots, a_{i,1}+a_{i,2}, \dots, a_{r}, b) 
- \psi_{r,s}(a_{1}, \dots, a_{i,1}, \dots, a_{r}, b) 
- \psi_{r,s}(a_{1}, \dots, a_{i,2}, \dots, a_{r}, b) \\
&= \psi_{r+1,s}(a_{1}, \dots, a_{i,1}, a_{i,2}, \dots, a_{r}, b).
\end{aligned}
\end{equation}
{\rm(ii)} Let $pr$ is the composition 
$\coprod_{F/k \ \text{finite}} A(F)^{\times r} \times K_{s}^{\mathrm{M}}(F) 
\to T_{r}(k; A, \mathcal{K}_{s}^{\mathrm{}M}) 
\twoheadrightarrow S_{r}(k; A, \mathcal{K}_{s}^{\mathrm{}M})$.
It holds that 
\begin{equation}                                                                                                             \label{property-6}
\Phi_{r,s}' \circ \psi_{r,s} = r! \cdot pr,
\end{equation}
where $\Phi_{r,s}'$ is defined after Proposition \ref{prop phi}.  
\end{lem}

\begin{proof}
The relation \eqref{property-1} follows from \eqref{elementary-1}.
By Lemma \ref{commutativity}, we have
\begin{align*}
\Phi_{r,s}' \psi_{r,s}(a_{1}, \dots, a_{r}, b)
&= \mathrm{Tr}_{F/k} \Phi_{r,s}'^{F} \Big( \big( ([a_{1}]_{F} - [0]_{F}) * \dots * ([a_{r}]_{F} - [0]_{F}) \big) \times [b]_{F} \Big) \\
&= \sum_{j=0}^{r} (-1)^{r-j} 
	\sum_{1\leq \nu_{1} < \dots < \nu_{j} \leq r}
	\left\{ \sum_{i = 1}^{j} a_{\nu_{i}}, \dots, \sum_{i = 1}^{j} a_{\nu_{i}}, b \right\}_{F/k} \\
&= \sum_{ \{ i_{1}, \dots, i_{r} \} = \{ 1, \dots, r \} } \left\{ a_{i_{1}}, \dots, a_{i_{r}}, b \right\}_{F/k} \\
&= r! \left\{ a_{1}, \dots, a_{r}, b \right\}_{F/k}.
\end{align*}
At the third equality we compute the coefficient of a symbol 
$\{ a_{i_{1}}, \dots, a_{i_{r}}, b \}_{F/k}$ with $i_{1}, \dots, i_{r} \in \{ 1, \dots, r \}$ 
that arises when the left hand side is developed. 
If the subset $\{ i_{1}, \dots, i_{r} \} \subset \{ 1, \dots, r \}$ consists of different $c$ ($1 \leq c \leq r$) elements,
it turns out to be $\sum_{j = 0}^{r-c} (-1)^{r-c-j} \binom{r-c}{j}$,
which is $0$ if $c < r$ and is $1$ if $c = r$.
\end{proof}

\begin{defn}																				\label{Gazaki type filtration}
We define two descending filtrations $\{F_{s}^{\nu}\}_{\nu \geq 0}$ and $\{G_{s}^{\nu}\}_{\nu \geq 0}$ 
of subgroups in $\mathrm{CH}^{d+s}(A, s)$. 
For $0 \leq \nu \leq s$, define
\begin{equation*}
F_{s}^{\nu} = G_{s}^{\nu} = \mathrm{CH}^{d+s}(A, s).
\end{equation*}
Let $r,s \geq 0$.
We define 
\begin{equation*}
F_{s}^{r+s} = \bigcap_{j=0}^{r-1} \mathrm{Ker} \, \Phi_{j,s}', \ 
G_{s}^{r+s} 
= \left\langle \mathrm{Im}(\psi_{r,s}) \right\rangle.
\end{equation*}
See Proposition \ref{prop phi} and Lemma \ref{lemma psi} for $\Phi_{r,s}'$ and $\psi_{r,s}$.
\end{defn}
It follows from \eqref{property-1} that $G_{s}^{\nu} \supset G_{s}^{\nu+1}$.
By definition, the map $\Phi_{r,s}'$ induces an injection 
\begin{equation}																					\label{injection}
F_{s}^{r+s}/F_{s}^{r+s+1} \hookrightarrow S_{r}(k; A, \mathcal{K}_{s}^{\mathrm{M}}).
\end{equation}
Let $\pi : A \to \mathrm{Spec} \, k$ be the structure morphism. For example, 
\begin{gather*}
F_{s}^{s+1} = \mathrm{Ker} (\Phi_{0,s} = \pi_{*} : \mathrm{CH}^{d+s}(A, s) \to 
K_{s}^{\mathrm{M}}(k) \cong \mathrm{CH}^{s}(k, s)), \\
G_{s}^{s+1} = \left\langle \mathrm{Tr}_{F/k}\Big( ([x]_{F} - [0]_{F}) \times [b]_{F} \Big) 
: x \in A(F), \, b \in K_{s}^{\mathrm{M}}(F) \right\rangle, \\
G_{s}^{s+2} = \left\langle \mathrm{Tr}_{F/k}\Big( ([x+y]_{F} - [x]_{F} - [y]_{F} + [0]_{F}) \times [b]_{F} \Big) 
: x, y \in A(F), \, b \in K_{s}^{\mathrm{M}}(F) \right\rangle.
\end{gather*}
It holds that $F_{s}^{s+1} = G_{s}^{s+1}$. 
		\begin{prop}                                                                                                                      \label{subfiltration}
The filtration $\{G_{s}^{\nu}\}_{\nu \geq 0}$ is a subfiltration of $\{F_{s}^{\nu}\}_{\nu \geq 0}$.
\end{prop}

\begin{proof}
For $\nu = 0, \dots, s$, the claim is trivial. 
Now we assume the claim for $\nu < r+s$ (for some $r$).
Then $\Phi_{j,s}'(G_{s}^{r+s}) \subset \Phi_{j,s}'(G_{s}^{r+s-1}) = 0$ for $j = 0, \dots, r-2$.
It is sufficient to show that $\Phi_{r-1,s}'(G_{s}^{r+s}) = 0$ for $r \geq 1$.
By \eqref{property-1} and \eqref{property-6}, we have
\begin{align*}
\Phi_{r-1,s}' \psi_{r,s} (a_{1}, \dots, a_{r}, b)
&= \Phi_{r-1,s}' \psi_{r-1,s} (a_{1}+a_{2}, a_{3}, \dots, a_{r}, b) \\
& \ \ \ - \Phi_{r-1,s}' \psi_{r-1,s} (a_{1}, a_{3}, \dots, a_{r}, b)
- \Phi_{r-1,s}' \psi_{r-1,s} (a_{2}, a_{3}, \dots, a_{r}, b) \\
&= (r-1)! \{ a_{1}+a_{2}, a_{3}, \dots, a_{r}, b \}_{F/k} \\
& \ \ \ - (r-1)! \{ a_{1}, a_{3}, \dots, a_{r}, b \}_{F/k}
- (r-1)! \{ a_{2}, a_{3}, \dots, a_{r}, b \}_{F/k} \\
&= 0.
\end{align*}
This concludes the proof.
\end{proof} 
		By \eqref{property-1}, we obtain a surjective homomorphism:
$$
T_{r}(k; A, \mathcal{K}_{s}^{\mathrm{M}}) \to 
\frac{G_{s}^{r+s}\mathrm{CH}^{d+s}(A, s)}{G_{s}^{r+s+1}\mathrm{CH}^{d+s}(A, s)},
$$
where $T_{r}(k; A, \mathcal{K}_{s}^{\mathrm{M}})$ is defined at \eqref{Tproduct}.
By Proposition \ref{subfiltration}, it induces a homomorphism:
$$
\psi_{r,s}' : T_{r}(k; A, \mathcal{K}_{s}^{\mathrm{M}}) \to 
\frac{F_{s}^{r+s}\mathrm{CH}^{d+s}(A, s)}{F_{s}^{r+s+1}\mathrm{CH}^{d+s}(A, s)}.
$$ 

\begin{prop}																				\label{Psi_{r,s}}
Let $r,s \geq 0$ be integers. The homomorphism $\psi_{r,s}'$ induces
\begin{align*}
\Psi_{r,s}' : S_{r}(k; A, \mathcal{K}_{s}^{\mathrm{M}}) &\to 
\frac{F_{s}^{r+s}\mathrm{CH}^{d+s}(A, s)}{F_{s}^{r+s+1}\mathrm{CH}^{d+s}(A, s)} \\
\{ a_{1}, \dots, a_{r}, b \}_{F/k} 
&\mapsto \mathrm{Tr}_{F/k} \big(
\left( \lambda_{F}(a_{1}) * \dots * \lambda_{F}(a_{r}) \right)
\times [b]_{F} \big), 
\end{align*}
and the property $\Phi_{r,s}' \circ \Psi_{r,s}' = r!$ holds on $S_{r}(k; A, \mathcal{K}_{s}^{\mathrm{M}})$.
\end{prop}

\begin{proof}
The property \eqref{property-6} shows that $\Phi_{r,s}' \circ \psi_{r,s}' = r! \cdot pr$, 
where $pr$ is the natural projection 
$T_{r}(k; A, \mathcal{K}_{s}^{\mathrm{M}}) \twoheadrightarrow S_{r}(k; A, \mathcal{K}_{s}^{\mathrm{M}})$.
Set $H_{i} = A$ for $1 \leq i \leq r$ and $H_{r+1} = \mathcal{K}_{s}^{\mathrm{M}}$. 
If $E/F/k$ are finite field extensions and we have $h_{i_{0}} \in H_{i_{0}}(E)$ for some $i_{0} \in \{1, \dots, r+1\}$, 
and $h_{i} \in H_{i}(F)$ for all $i \neq i_{0}$, then
\begin{align*}
\Phi_{r,s}' \circ \psi_{r,s}'(h_{1} \otimes \dots \otimes \mathrm{Tr}_{E/F}(h_{i_{0}}) \otimes \dots \otimes h_{r+1}) 
&= r! \{ h_{1}, \dots, \mathrm{Tr}_{E/F}(h_{i_{0}}), \dots, h_{r+1} \}_{F/k} \\
&= r! \{ \mathrm{R}_{E/F}(h_{1}), \dots, h_{i_{0}}, \dots, \mathrm{R}_{E/F}(h_{r+1}) \}_{E/k} \\
&= \Phi_{r,s}' \circ \psi_{r,s}'
(\mathrm{R}_{E/F}(h_{1}) \otimes \dots \otimes h_{i_{0}} \otimes \dots \otimes \mathrm{R}_{E/F}(h_{r+1}))
\end{align*}
By \eqref{injection}, we have
$$
\psi_{r,s}'(h_{1} \otimes \dots \otimes \mathrm{Tr}_{E/F}(h_{i_{0}}) \otimes \dots \otimes h_{r+1}) 
= \psi_{r,s}'(\mathrm{R}_{E/F}(h_{1}) \otimes \dots \otimes h_{i_{0}} \otimes \dots \otimes \mathrm{R}_{E/F}(h_{r+1})).
$$

Let $K$ be a function field in one variable over $k$. 
Let $f_{1}, \dots, f_{r} \in A(K)$ and $g \in K_{s+1}^{\mathrm{M}} (K)$. Then
\begin{align*}
&\psi_{r,s}' \left( \sum_{v} s_{v}(f_{1}) \otimes \dots \otimes s_{v}(f_{r}) \otimes \partial_{v} (g) \right) \\
&= \sum_{j=0}^{r} (-1)^{r-j} \sum_{1\leq \nu_{1} < \dots < \nu_{j} \leq r} 
	\sum_{v} \mathrm{Tr}_{k(v)/k} \left(
	[s_{v}(f_{\nu_{1}} + \dots + f_{\nu_{j}})]_{k(v)} \times [\partial_{v} (g)]_{k(v)} \right) \\
&= 0, 
\end{align*}
where the last equality follows by Theorem \ref{1.4.} and the following commutative diagram: 
$$
\begin{tikzcd}
	A(K) \arrow[r, "s_{v}"] \arrow[d, "\text{[ ]}_{K}"'] 
	& A(k(v)) \arrow[d, "\text{[ ]}_{k(v)}"] \\ 
	\mathrm{CH}_{0}(A_{K}) \arrow[r, "s_{v}"'] 
	& \mathrm{CH}_{0}(A_{k(v)}).
\end{tikzcd}
$$
Thus $\Psi_{r,s}'$ is well-defined and the latter statement is concluded.
\end{proof} 
		\begin{cor}
The composition
$$
\Psi_{r,s}' \circ \Phi_{r,s}' : (G_{s}^{r+s} + F_{s}^{r+s+1}) / F_{s}^{r+s+1} 
\to S_{r}(k; A, \mathcal{K}_{s}^{\mathrm{M}}) 
\to (G_{s}^{r+s} + F_{s}^{r+s+1}) / F_{s}^{r+s+1}
$$
is the multiplication by $r!$.
\end{cor}

\begin{proof}
The subgroup $(G_{s}^{r+s} + F_{s}^{r+s+1}) / F_{s}^{r+s+1} \subset F_{s}^{r+s} / F_{s}^{r+s+1}$ is the image of $\Psi_{r,s}'$, 
so that the claim is deduced from $\Psi_{r,s}' \Phi_{r,s}' \Psi_{r,s}' = r! \Psi_{r,s}'$.
\end{proof}

\begin{thm}                                                                                                                         \label{canonical isomorphisms}
The injection \eqref{injection} is an isomorphism up to $r!$-torsion: 
$$
\Phi_{r,s}' : F_{s}^{r+s} / F_{s}^{r+s+1} \otimes \mathbb{Z} \left[ \frac{1}{r!} \right] \xrightarrow{\sim} 
S_{r}(k; A, \mathcal{K}_{s}^{\mathrm{M}}) \otimes \mathbb{Z} \left[ \frac{1}{r!} \right]
$$
with $\Phi_{r,s}'^{-1} = (1/r!)\Psi_{r,s}'$.
\end{thm}

\begin{proof}
The multiplication by $r!$ is an isomorphism after $\otimes \mathbb{Z} [ 1 / r! ]$. 
Therefore \eqref{injection} is also surjective after $\otimes \mathbb{Z} [ 1 / r! ]$ 
by Proposition \ref{Psi_{r,s}}.
\end{proof}

%
\section{The \'{e}tale cycle map and the Somekawa map}
In addition to the setting in the section \ref{section Gazaki type}, 
we use the following notations in this section.

\noindent \textit{Notations.}
Throughout this section, fix an integer $n > 0$ invertible in $k$.
For a $\mathbb{Z}$-module $M$ and an integer $m$, let $M[m] := \mathrm{Ker}(M \xrightarrow{m} M)$,
and for an integer $r \geq 0$, denote by $\bigwedge^{r}M$ the $r$-th exterior product,
which is the quotient of $\bigotimes^{r} M$ 
by the submodule generated by elements $x_{1} \otimes \dots \otimes x_{r}$ 
in which two of them are equal. 
Let $\mathbb{Z} / n (1) = \mu_{n} := \mathbb{G}_{\mathrm{m}}[n]$.
	\subsection{The Somekawa map}														\label{Somekawa map}
		For a semi-abelian variety $G$ over $k$, we have the Kummer exact sequence 
$0 \to G[n] \to G \xrightarrow{n} G \to 0$.
For an extension $F/k$ of fields, we denote the connecting homomorphism by 
\begin{equation}																			\label{connecting}
\delta : G(F) \to H_{\text{\'{e}t}}^{1}(F, \, G_{F}[n]).
\end{equation}

Let $G_{1}, \dots, G_{r}$ be semi-abelian varieties over $k$.
In \cite{S}, Somekawa defines the morphism
\begin{align*}
s_{n} : \frac{K(k; G_{1}, \dots, G_{r})}{n} &\to H_{\text{\'{e}t}}^{r}(k, G_{1}[n] \otimes \dots \otimes G_{r}[n]) \\
\{ a_{1}, \dots, a_{r} \}_{F/k} &\mapsto \mathrm{Tr}_{F/k} (\delta(a_{1}) \cup \dots \cup \delta(a_{r}))
\end{align*}
where $F$ is a finite extension of $k$ and $a_{i} \in G_{i}(F)$.

When $G_{1} = \dots = G_{r} = \mathbb{G}_{\mathrm{m}}$, this gives the Galois symbol
\begin{equation}																			\label{Galois symbol}
h_{k,n} : K_{r}^{\mathrm{M}}(k) / n \to H^{r}(k, \mu_{n}^{\otimes r})
\end{equation}
sending a symbol $\{ b_{1}, \dots, b_{r} \}$ to $\delta(b_{1}) \cup \dots \cup \delta(b_{r})$ 
for $b_{1}, \dots, b_{r} \in k^{*}$ by \cite[Theorem 1.4]{S}.
\begin{thm} (Rost-Voevodsky, \cite[Theorem 6.16]{V})									\label{RV}
The Galois symbol \eqref{Galois symbol} is an isomorphism for any $n$ invertible in $k$.
\end{thm}

Let $r,s \geq 0$ be two integers.
Let $p$ be the natural projection $p : A[n]^{\otimes r} \to \bigwedge^{r}A[n]$ and 
set $q = p \otimes \mathrm{id}^{\otimes s} : 
A[n]^{\otimes r} \otimes \mu_{n}^{\otimes s} \to 
\bigwedge^{r}A[n] \otimes \mu_{n}^{\otimes s}$.
\begin{prop}																			\label{S map type-1}
With notation as above, the composition $q_{*} \circ s_{n}$ induces
$$
s_{n} : \frac{S_{r,s}(k; A, \mathbb{G}_{\mathrm{m}})}{n} \to H^{r+s}(k, \bigwedge^{r}A[n] \otimes \mu_{n}^{\otimes s}).
$$
The group $S_{r,s}(k; A, \mathbb{G}_{\mathrm{m}})$ is defined in Definition \ref{SomekawaK}.
\end{prop}

The proof is exactly parallel to \cite[Proposition 5.2]{G} and we omit it.

\begin{cor}																		\label{Cor Somekawa type map}
Assume that $k$ is perfect.
The homomorphism 
\begin{align*}
s_{n}' : T_{r}(k; A, \mathcal{K}_{s}^{\mathrm{M}}) &\to H^{r+s}(k, A[n]^{\otimes r} \otimes \mu_{n}^{\otimes s}) \\
(a_{1} \otimes \dots \otimes a_{r} \otimes b)_{F} &\mapsto 
\mathrm{Tr}_{F/k} ( \delta(a_{1}) \cup \dots \cup \delta(a_{r}) \cup h_{F,n}(b) )
\end{align*}
where $F/k$ is a finite extension and $a_{1}, \dots, a_{r} \in A(F)$, $b \in K_{s}^{\mathrm{M}}(F)$, 
factors through $K_{r}(k; A, \mathcal{K}_{s}^{\mathrm{M}}) / n$.
Furthermore, 
$$
s_{n}' : \frac{S_{r}(k; A, \mathcal{K}_{s}^{\mathrm{M}})}{n} \to 
H^{r+s}(k, \bigwedge^{r}A[n] \otimes \mu_{n}^{\otimes s})
$$
is induced.
\end{cor}

When $s = 0,1$ or $r=0$, the morphism \eqref{two somekawa} is an isomorphism. 
Hence we do not need an assumption of perfectness in these cases. 

\begin{proof}
The desired morphisms $s_{n}'$ are obtained by
\begin{align*}
&K_{r}(k; A, \mathcal{K}_{s}^{\mathrm{M}}) \xleftarrow{\sim} 
K_{r,s}(k; A, \mathbb{G}_{\mathrm{m}}) \xrightarrow{s_{n}} 
H^{r+s}(k, A[n]^{\otimes r} \otimes \mu_{n}^{\otimes s}), \\
&S_{r}(k; A, \mathcal{K}_{s}^{\mathrm{M}}) \xleftarrow{\sim} 
S_{r,s}(k; A, \mathbb{G}_{\mathrm{m}}) \xrightarrow{s_{n}} 
H^{r+s}(k, \bigwedge^{r}A[n] \otimes \mu_{n}^{\otimes s}).
\end{align*}
The first isomorphism is due to Kahn and Yamazaki (Theorem \ref{KY}).
\end{proof} 
	\subsection{The Hochschild-Serre spectral sequence}                            \label{subsection H-S}
		In this section, we fix an integer $t$ and use the following notations.
Let $\bar{k}$ be a separable closure of $k$.
For a scheme $X$ over $k$, denote $\overline{X} := X_{\bar{k}}$. 
For an \'{e}tale sheaf $\mathcal{F}$ of $\mathbb{Z} / n$-modules on $X_{\text{\'{e}t}}$, 
denote $\mathcal{F}(t) := \mathcal{F} \otimes \mu_{n}^{\otimes t}$. 
We put $0^{0} = 1$ and $(-1)! = 1$ by conventions.

We consider the Hochschild-Serre spectral sequence
\begin{equation}                                                                                                          \label{H-S}
E_{2}^{i, j} = H^{i}(k, H^{j}(\bar{A}, \, \mathbb{Z}/n (t))) \Rightarrow H^{i+j}(A, \, \mathbb{Z}/n (t))
\end{equation}
One has
\begin{equation}                                                                                                                  \label{E_2}
E_{2}^{i,j} \cong H^{i}(k, \bigwedge^{2d-j} A[n] (t-d) )
\end{equation}
by the Poincar\'{e} duality and Theorem 12.1 in \cite{M}. 

Let $m$ be an integer. Consider the multiplication morphism $m : A \to A$.
The pull-back $m^{*}$ acts on $H^{j}(\bar{A}, \mathbb{Z} / n(t))$ 
as the multiplication by $m^{j}$. 
The push-forward $m_{*}$ acts on $H^{j}(\bar{A}, \mathbb{Z} / n(t))$ 
as the multiplication by $m^{2d-j}$.

\begin{prop} 																					\label{degeneration}
Assume that the condition $l > \min(\mathrm{cd}_{l}(k), 2d+1) =: M$ holds for any prime number $l$ dividing $n$.
Then the Hochschild-Serre spectral sequence \eqref{H-S}
degenerates at level two.
\end{prop}
\begin{remark}																			\label{Remark 1}
Proposition \ref{degeneration} generalizes {\rm \cite[Lemma 6.6]{G}},
where Gazaki showed the statement when k is a finite extension of $\mathbb{Q}_{p}$ (then $cd(k) = 2$). 
Our proof extends her arguments.
\end{remark}

\begin{proof}
We may assume that $n = l^{e}$ for a prime number $l$ satisfying $l > M$.
We will show that the boundary map $d_{r}^{i, j} : E_{r}^{i, j} \to E_{r}^{i+r, j-r+1}$ is zero map
for every $r \geq 2$, $i$ and $j$.
If $r \geq M+1$, this claim is trivial simply because the domain or the target is zero.
Let $2 \leq r \leq M$. Let $m$ be any integer and consider the multiplication morphism $m : A \to A$.
The pull-back $m^{*}$ acts on $E_{r}^{i,j}$ as the multiplication by $m^{j}$. 
Since the pull-back $m^{*}$ and $d_{r}^{i,j}$ are compatible, that is, $d_{r}^{i,j} m^{j} = m^{j-r+1} d_{r}^{i,j}$,
we obtain an equality $m^{j-r+1}(m^{r-1} - 1) d_{r}^{i,j} = 0$. 
Choose $m$ to be a $(l-1)$-th primitive root of unity.
Then for every $2 \leq r \leq M$, it holds that $m^{j-r+1}(m^{r-1} - 1) \in (\mathbb{Z} / l)^{\times}$ by $M < l$, 
hence $d_{r}^{i, j} = 0$.
\end{proof} 
		From the spectral sequence \eqref{H-S}, 
we obtain a descending filtration
$$
H^{q}(A, \, \mathbb{Z}/n (t)) = {\rm fil}_{{\rm HS}}^{0}H^{q} \supset {\rm fil}_{{\rm HS}}^{1}H^{q} \supset \dots \supset {\rm fil}_{{\rm HS}}^{q}H^{q} \supset 0
$$
with quotients ${\rm fil}_{{\rm HS}}^{\nu}H^{q} / {\rm fil}_{{\rm HS}}^{\nu + 1}H^{q} \simeq E_{\infty}^{\nu, q - \nu}$.
Let $s \geq 0$ be an integer. 
If $\nu < s$, it holds that $E_{2}^{\nu, 2d+s-\nu} = 0$.
We obtain $H^{2d+s}(A, \mathbb{Z}/n(d+s)) = {\rm fil}_{{\rm HS}}^{0}H^{2d+s} = \dots = {\rm fil}_{{\rm HS}}^{s}H^{2d+s}$.

Let $r,s \geq 0$ be two integers. 
For proper smooth connected $k$-schemes $X, Y$,
\begin{itemize}
\item
Set 
\begin{equation}																			\label{L-functor}
L_{s}(X) := \bigoplus_{x \in X_{(0)}} H^{s}(x, \mathbb{Z}/n(s))
\end{equation}
and denote by $(b)_{x} \in L_{s}(X)$ the element determined by $x \in X_{(0)}$ and $b \in H^{s}(x, \mathbb{Z}/n(s))$.
For a $k$-morphism $f : X \to Y$, define $f_{*} : L_{s}(X) \to L_{s}(Y)$ 
by $f_{*}((b)_{x}) = ({\rm N}_{k(x) / k(y)} b)_{y}$, where $y = f(x)$.
By this correspondence $L_{s}$ is a covariant functor.
\item
Let $d$ be the dimension of $X$.
Set $\psi_{X} := \sum_{x \in X_{(0)} } x_{*} : L_{s}(X) \to H^{2d+s}(X, \mathbb{Z}/n(d+s))$, 
which is natural in $X$.
We define a descending filtration on $L_{s}(X)$ 
by ${\rm fil}^{\nu} L_{s}(X) := \psi_{X}^{-1}({\rm fil}_{{\rm HS}}^{\nu}H^{2d+s}(X))$. 
\item
For an abelian variety $X = A$, we define 
\begin{equation}
c_{r,s} : L_{s}(A) \to H^{r+s}(k, \bigwedge^{r} A[n] (s) )
\end{equation}
by $c_{r,s}( (b)_{x} ) = \mathrm{Cor}_{k(x) / k} (\underbrace{\delta(x) \cup \dots \cup \delta(x)}_{r} \cup b)$ 
for $x \in A_{(0)}$ and $b \in H^{s}(x, \mathbb{Z}/n(s))$,
where $\delta : A(k(x)) \to H^{1}(k(x), A[n])$ is from \eqref{connecting}.

\end{itemize}
\begin{lem}                                                                                          \label{hs&cup}
Let $r,s \geq 0$ be two integers. 
Assume that the spectral sequence \eqref{H-S} degenerates at level two.
Let $pr_{r,s} : {\rm fil}_{{\rm HS}}^{r+s}H^{2d+s}(A, \, \mathbb{Z}/n (d+s)) \twoheadrightarrow E_{2}^{r+s, 2d-r}$
be the natural projection.
Then the following diagram is commutative:
\begin{center}
\begin{tikzcd}
{\rm fil}_{{\rm HS}}^{r+s}H^{2d+s}(A, \, \mathbb{Z}/n (d+s)) \arrow[r, "pr_{r,s}"]
&H^{r+s}(k, \bigwedge^{r} A[n] (s)) \arrow[r, "r^{r}"]
&H^{r+s}(k, \bigwedge^{r} A[n] (s)) \\
{\rm fil}^{r+s} L_{s}(A) \arrow[u, "\psi_{A}"]\arrow[from=2-1, to=2-3, hook]
&
&L_{s}(A) \arrow[u, "c_{r,s}"'].
\end{tikzcd}
\end{center}
\end{lem}

We use the following notation in the proof. 
Let $C$ be a proper smooth connected curve over $k$ of genus $g$ and fix a base point $O \in C(k)$.
Put $J = {\rm Jac} (C)$. 
Let $\varphi^{O} : C \to J$ be the Abel-Jacobi map 
such that $\varphi^{O}(O) = 0_{J}$ \cite[Chapter III, Section 2]{M}.
Denote the $r$-fold product of $C$ by $C^{r} = C \times \dots \times C$.
Let $\varphi_{r}^{O} : C^{r} \to J$ be the map sending 
$(P_{1}, \dots, P_{r})$ to $\varphi^{O}(P_{1}) + \dots + \varphi^{O}(P_{r})$.
We simply denote $H^{j}(X, \mathbb{Z}/n (t))$ by $H^{j}(X, t)$.
Consider the Hochschild-Serre spectral sequence
$$
E_{2}^{i,j} = H^{i}(k, H^{j}(\overline{C}^{r}, t)) \Rightarrow H^{i+j}(C^{r}, t).
$$
This sequence induces a direct sum decomposition \cite[Proposition 2.4]{Y}:
\begin{equation}																				\label{E2-split}
H^{m}(C^{r}, t) \cong \bigoplus_{i + j = m} H^{i}(k, H^{j}(\overline{C}^{r}, t)),
\end{equation}
which depends on the choice of $O \in C(k)$.
Define a map 
\begin{equation}																				\label{fact morphism}
f_{r,s}^{O} : L_{s}(C^{r}) \to H^{r+s}(k, \bigwedge^{r}J[n] (s) )
\end{equation}
by $f_{r,s}^{O}((b)_{x}) = \mathrm{Cor}_{k(x) / k} (\delta( \varphi^{O}(x_{1}) ) \cup \dots \cup \delta( \varphi^{O}(x_{r}) ) \cup b)$
for $x \in C^{r}_{(0)}$ and $b \in H^{s}(x, s)$. 
Here $x_{i} \in C(k(x))$ is the projection of $x \in C^{r}(k(x))$ to the $i$-th component.
Then the composition map
\begin{center}
\begin{tikzcd}
L_{s}(C^{r}) \arrow[r, "\psi_{C^{r}}"] 
&H^{2r+s} (C^{r}, r+s) \arrow[r, twoheadrightarrow, "pr_{r,s}"] 
&E_{2}^{r+s, r} = H^{r+s}(k, H^{r}(\overline{C}^{r}, r+s)) \arrow[r, "\overline{\varphi_{r*}^{O}}"]
&H^{r+s}(k, \bigwedge^{r} J[n] (s) )
\end{tikzcd}
\end{center}
agrees with $f_{r,s}^{O}$ (cf. \cite[Proposition 2.4.\ Proof]{Y}). 
Here $pr_{r,s}$ is the natural projection associated to \eqref{E2-split}
and $\overline{\varphi_{r*}^{O}}$ is induced by 
$(\overline{\varphi_{r}^{O}})_{*} : H^{r}(\overline{C}^{r}, r+s) \to H^{2g-r}(\overline{J}, g+s) = \bigwedge^{r}J[n] (s)$.
\begin{proof}
By a standard norm argument, we may assume $k$ is an infinite field. 
Take an element $\alpha \in \mathrm{fil}^{r+s}L_{s}(A)$ 
and write $\alpha = \sum_{i=1}^{N} (b_{i})_{x_{i}}$, $x_{i} \in A_{(0)}$ and  $b_{i} \in \mathrm{H}^{s}(x_{i}, s)$.
By Bertini's theorem (using an assumption of infiniteness of $k$), 
there is a smooth projective connected curve $C \subset A$
containing the origin point $0_{A}$ of A and $x_{1}, \dots, x_{N}$.
Then $\alpha$ is contained in $L_{s}(C)$, which is a direct summand of $L_{s}(A)$.
The closed immersion $i$ of $C$ factors uniquely as follows:
\begin{equation}																				\label{abel-jacobi}
\begin{tikzcd} [column sep = small]
C \arrow[rr, hook, "i"]\arrow[dr, "\varphi^{0_{A}}"'] &                                   & A \\
                                  & J \arrow[ur, "i_{J}"'] &
\end{tikzcd}
\end{equation}
where $i_{J}$ is a homomorphism of group schemes \cite[Chapter III, Proposition 6.1]{M}.
Let $i_{r} : C^{r} \to A$ be the composition of 
$i \times \dots \times i : C^{r} \to A^{r}$
and the multiplication $A^{r} \to A$. 
Let $\delta_{r} : C \hookrightarrow C^{r}$ be the diagonal embedding.
We have $i_{r} = i_{J} \circ \varphi_{r}^{0_{A}}$ and $r \circ i = i_{r} \circ \delta_{r}$.
We obtain the following commutative diagram
\begin{equation}																			\label{wonderful}
\begin{tikzpicture}
\node[xscale=0.9] (a) at (0,0) {
\begin{tikzcd} [column sep = 1.6em]
\mathrm{fil}^{r+s}L_{s}(C) \arrow[r, "\delta_{r*}"]\arrow[d, hook, "i_{*}"']
&\mathrm{fil}^{r+s}L_{s}(C^{r}) \arrow[r, "\psi_{C^{r}}"]\arrow[d, "i_{r*}"]
&{\rm fil}_{{\rm HS}}^{r+s}H^{2r+s}(C^{r}, \, r+s) \arrow[r, "pr_{r,s}"]\arrow[d, "i_{r*}"]
&H^{r+s}(k, H^{r}(\overline{C}^{r}, r+s)) \arrow[r, "\overline{\varphi_{r*}}"]\arrow[d, "\overline{i_{r*}}"]
&H^{r+s}(k, \bigwedge^{r}J[n] (s)) \arrow[dl, "\overline{i_{J*}}"] \\
\mathrm{fil}^{r+s}L_{s}(A) \arrow[r, "r_{*}"]
&\mathrm{fil}^{r+s}L_{s}(A) \arrow[r, "\psi_{A}"]
&{\rm fil}_{{\rm HS}}^{r+s}H^{2d+s}(A, d+s) \arrow[r, "pr_{r,s}"]
&H^{r+s}(k, \bigwedge^{r} A[n] (s)).
\end{tikzcd}
};
\end{tikzpicture}
\end{equation}
Therefore we have
\begin{align*}
r^{r} \cdot pr_{r,s} \circ \psi_{A} (\alpha) 
&= r_{*} \circ pr_{r,s} \circ \psi_{A} (\alpha) &( r_{*} = r^{r} \text{\ on\ } \bigwedge^{r} A[n] ) \\
&= pr_{r,s} \circ \psi_{A} \circ r_{*} \circ i_{*} (\alpha) &( r_{*} \text{\ commutes with $pr_{r,s}$ and $\psi_{A}$} ) \\
&= \overline{i_{J*}} \circ f_{r,s} \circ \delta_{r*} (\alpha) &( \text{\eqref{wonderful} and \eqref{fact morphism}} ) \\
&= \sum_{i} \overline{i_{J*}} \left( \mathrm{Cor}_{k(x_{i}) / k} (\delta( \varphi(x_{i}) ) \cup \dots \cup \delta( \varphi(x_{i}) ) \cup b_{i}) \right) 
&( \delta_{r*} ( (b)_{x} ) = (b)_{x \times \dots \times x} ) \\
&= \sum_{i} \mathrm{Cor}_{k(x_{i}) / k} (\delta( x_{i} ) \cup \dots \cup \delta( x_{i} ) \cup b_{i}) \\
&= c_{r,s} (\alpha) &( \text{definition of $c_{r,s}$} )
\end{align*}
where the second equality from the bottom we need the compatibility of $\mathrm{Cor}$ and $\overline{i_{J*}}$, 
that $\overline{i_{J*}} = ( \bigwedge^{r} i_{J} (s) )_{*}$ on $H^{r+s}(k, \bigwedge^{r}J[n] (s))$, 
and $i_{J*} \delta( \varphi(x) ) =  \delta( x )$ by \eqref{abel-jacobi}.
\end{proof} 
	\subsection{The cycle map}									\label{cycle-somekawa}
		Denote the cycle map by
$$
\rho_{A, n}^{s} : \mathrm{CH}^{d+s}(A, s) \to H^{2d+s}(A, \mathbb{Z}/n (d+s))
$$
and the cycle map modulo $n$ by
$$
\rho_{A}^{s} / n : \mathrm{CH}^{d+s}(A, s) / n \to H^{2d+s}(A, \mathbb{Z}/n (d+s)).
$$
We have the Gazaki type filtration
$\{ F_{s}^{\nu} \}_{\nu \geq 0}$ defined in Definition \ref{Gazaki type filtration}
and $\{ \mathrm{fil}_{{\rm HS}}^{\nu}H^{2d+s} \}_{\nu \geq 0}$ 
induced by the spectral sequence \eqref{H-S}.

\begin{thm}																			\label{admiring theorem}
Let $r, s \geq 0$ be two integers. 
Assume that $k$ is perfect, the spectral sequence \eqref{H-S} degenerates at level two, 
and $n$ is coprime to $(r-1)!$. 
Then we have
$$
\rho_{A, n}^{s} (F_{s}^{r+s}) \subset \mathrm{fil}_{{\rm HS}}^{r+s}H^{2d+s}.
$$
\end{thm}

\begin{proof}
Set $C_{s} (A) := \bigoplus_{x \in A_{(0)}} \mathrm{CH}^{s}(x, s)$ 
and $\phi_{A} := \sum_{x \in A_{(0)}} x_{*} : C_{s}(A) \to \mathrm{CH}^{d+s}(A, s)$.
The map $\phi_{A}$ is surjective.
We define a descending filtration $\{ F^{\nu}C_{s}(A) \}_{\nu \geq 0}$ by $F^{\nu}C_{s}(A) := \phi_{A}^{-1}(F_{s}^{\nu})$.
Put $\rho_{n}' := \bigoplus_{x \in A_{(0)}} \rho_{x,n} : C_{s}(A) \to L_{s}(A)$, 
where $L_{s}(A)$ is from \eqref{L-functor}.
We have the following commutative diagram:
\begin{equation}																				\label{rho&h}
\begin{tikzcd} [column sep = small]
C_{s}(A) \arrow[rr, "\sim"]\arrow[dr, "\rho_{n}'"'] 
&&\bigoplus_{x \in A_{(0)}} K_{s}^{\mathrm{M}} (k(x)) \arrow[dl, "h_{n}"] \\
&L_{s}(A).&
\end{tikzcd}
\end{equation}
See \eqref{Galois symbol} and \eqref{N-S,T} for $h_{n}$ and the top isomorphism.
Let $b_{r,s} : \bigoplus_{x \in A_{(0)}} K_{s}^{\mathrm{M}} (k(x)) \to S_{r,s}(k;A, \mathcal{K}_{s}^{\mathrm{M}})$
be a map sending $b \in K_{s}^{\mathrm{M}} (k(x))$ 
to $\{ x, \dots, x, b\}_{k(x) / k}$.
Now we will show our claim by induction on $r$. The case $r = 0$ is trivial.
Assume that it is correct for $r$. Then we have a diagram
\begin{equation}																				\label{admiring diagram}
\begin{tikzcd}
F^{r+s}C_{s}(A) \arrow[rrr, hook]\arrow[ddd, "\rho_{n}'"']\arrow[dr, twoheadrightarrow, "\phi_{A}"] 
&&&\bigoplus_{x \in A_{(0)}} K_{s}^{\mathrm{M}}(k(x)) \arrow[dl, "b_{r,s}"']\arrow[ddd, "h_{n}"] \\
&F_{s}^{r+s} \arrow[r, "\Phi_{r,s}'"]\arrow[d, "\rho_{A,n}"'] &S_{r}(k;A, \mathcal{K}_{s}^{\mathrm{M}}) \arrow[d, "s_{n}'"]& \\
&{\rm fil}_{{\rm HS}}^{r+s}H^{2d+s} \arrow[r, "r^{r} \cdot pr_{r,s}"'] 
&H^{r+s}(k, \bigwedge^{r}A[n] (s))& \\
{\rm fil}^{r+s}L_{s}(A) \arrow[ur, "\psi_{A}"]\arrow[rrr, hook] &&&L_{s}(A) \arrow[ul, "c_{r,s}"']
\end{tikzcd}
\end{equation}
where $s_{n}'$ is obtained in Corollary \ref{Cor Somekawa type map}.
The commutativity of the inner square follows from the surjectivity of $\phi_{A}$ 
and the commutativity of the surrounding four squares and the outer square. 
It is verified by Lemma \ref{hs&cup} for the lower square and by \eqref{rho&h} for the outer square.
\end{proof}

\begin{cor}																					\label{evaluation kernel}
In the situation of Theorem \ref{admiring theorem}, the following holds.
\begin{enumerate}
\item $\mathrm{Ker} (\rho_{A}^{s}/n) \subset (F_{s}^{s+1} + nF_{s}^{0}) / nF_{s}^{0}$ for $s \geq 0$, 
and $\mathrm{Ker} (\rho_{A}^{0}/n) \subset (F_{0}^{2} + nF_{0}^{0}) / nF_{0}^{0}$.
\item Let $s = 0, 1$. We consider the condition 
\begin{equation*}																			\label{inj-condition}
(*)_{s}:\ \text{the Somekawa map } 
K_{2-s,s}(k; A, \mathbb{G}_{\text{m}}) / n \xrightarrow{s_{n}} 
H^{2}(k, A[n]^{\otimes 2-s}\otimes \mu_{n}^{\otimes s}) \text{ is injective.}
\end{equation*}
If $(*)_{s}$ holds, then 
$\mathrm{Ker} (\rho_{A}^{s}/n) \otimes \mathbb{Z}[1/(2-s)] \subset 
(F_{s}^{3} + nF_{s}^{0}) / nF_{s}^{0} \otimes \mathbb{Z}[1/(2-s)]$.
\end{enumerate}
\end{cor}

\begin{proof}
Let $\pi : A \to \mathrm{Spec} \, k$ be the structure morphism of $A$.
We have the following commutative diagram:
\begin{equation}                                                                                                           \label{diagram-1}
\begin{tikzcd}
\mathrm{CH}^{d+s}(A, s) \arrow[r, twoheadrightarrow]\arrow[d, "\rho_{A,n}"']
\arrow[from=1-1, to=1-4, bend left=15, twoheadrightarrow, "\pi_{*}"]
&F_{s}^{s}/F_{s}^{s+1}
\arrow[r, "\Phi_{0,s}", "\sim"']
&K_{s}^{\mathrm{M}}(k) \arrow[r, "\text{[ ]}_{k}", "\sim"']\arrow[d, twoheadrightarrow, "s_{n} = h_{n}"]
&\mathrm{CH}^{s}(k, s) \arrow[dl, "\rho_{k,n}"] \\
H^{2d+s}(A, \mathbb{Z}/n(d+s)) \arrow[r, twoheadrightarrow]
\arrow[from=2-1, to=2-3, bend right=15, twoheadrightarrow, "pr_{0,s} = \pi_{*}"']
&{\rm fil}_{{\rm HS}}^{s}H^{2d+s}/{\rm fil}_{{\rm HS}}^{s+1}H^{2d+s} \arrow[r, "\sim"]
&H^{s}(k, \mu_{n}^{\otimes s})
\end{tikzcd}
\end{equation}
where $h_{n}$ and $[ \ ]_{k}$ are from \eqref{Galois symbol} and \eqref{N-S,T}. 
By the injectivity of \eqref{Galois symbol}, 
one verifies $\mathrm{Ker} (\rho_{A}^{s}/n) \subset (F_{s}^{s+1} + nF_{s}^{0}) / nF_{s}^{0}$.
This proves the first statement of (i).

We have commutative diagrams for $(s, r) = (0, 1), (0, 2), (1, 1)$:
\begin{equation}																					\label{diagram-2}
\begin{tikzcd}
F_{s}^{r+s} \arrow[r, twoheadrightarrow]\arrow[d, "\rho_{A,n}"']
&F_{s}^{r+s} / F_{s}^{r+s+1} \arrow[r, hook, "\Phi_{r,s}'"]
&S_{r,s}(k; A, \mathbb{G}_{\mathrm{m}}) \arrow[d, "s_{n}"] \\
{\rm fil}_{{\rm HS}}^{1}H^{2d} \arrow[rr, "r^{r}\cdot pr_{r,s}"']
&
&H^{r+s}(k, \bigwedge^{r}A[n](s))
\end{tikzcd}
\end{equation}
where $s_{n}$ is obtained in Proposition \ref{S map type-1}. 
Recall \eqref{two somekawa}.
When $(s, r) = (0, 1)$, the map $s_{n}$ agrees with $\delta : A(k) \to H^{1}(k, A[n])$ from \eqref{connecting}. 
Since $A(k)/n \xrightarrow{\delta} H^{1}(k, A[n])$ is injective, 
$\mathrm{Ker} (\rho_{A}^{0}/n) \subset (F_{0}^{2} + n\mathrm{CH}_{0}(A))/n\mathrm{CH}_{0}(A)$.
This proves the second statement of (i).
When $(s, r) = (0, 2)$, Gazaki has shown in \cite[Proposition 6.8.~Proof]{G} that 
$s_{n} : S_{2}(k;A)/n \otimes \mathbb{Z}[1/2] \to H^{2}(k, \bigwedge^{2} A[n]) \otimes \mathbb{Z}[1/2]$
is injective if the condition $(*)_{0}$ is satisfied. 
This shows the statement (ii) for $s=0$: 
$\mathrm{Ker} (\rho_{A}^{0}/n) \otimes \mathbb{Z}[1/2] \subset 
(F_{0}^{3} + n\mathrm{CH}_{0}(A))/n\mathrm{CH}_{0}(A) \otimes \mathbb{Z}[1/2]$.
Similarly, the statement (ii) for $s = 1$ follows.
\end{proof}

\begin{remark}																					\label{remark 2}
In this remark let $k$ be a finite extension of $\mathbb{Q}_{p}$ and $n > 0$ any integer (which may be even).
Consider the spectral sequence \eqref{H-S}.
Since the $p$-adic field has cohomological dimension two, one has ${\rm fil}_{{\rm HS}}^{3} = 0$.
If $n$ is odd, then \eqref{H-S} degenerates at level two (see Remark \ref{Remark 1}).
We have $E_{\infty}^{2, 2d-1} = E_{2}^{2,2d-1}/\mathrm{Im} \ d_{2}^{0, 2d}$. 
Let $pr : E_{2}^{2,2d-1} \twoheadrightarrow E_{\infty}^{2, 2d-1}$ be the natural projection, 
which is an isomorphism if $n$ is odd. 
Even if $n$ is even, we have the following commutative diagram as the proof of Theorem \ref{admiring theorem}:
\begin{equation}																						\label{diagram-4}
\begin{tikzcd}
F_{1}^{2} \arrow[r, twoheadrightarrow]\arrow[d, "\rho_{A,n}"'] 
&F_{1}^{2}/F_{1}^{3} \arrow[r, "\Phi_{1,1}", "\sim"'] 
&K(k; A, \mathbb{G}_{\mathrm{m}}) \arrow[d, "s_{n}"] \\
{\rm fil}_{{\rm HS}}^{2}H^{2d+1} \arrow[r, twoheadrightarrow, "pr_{1,1}"'] 
&E_{\infty}^{2, 2d-1}
&H^{2}(k, A[n] \otimes \mu_{n}). \arrow[l, twoheadrightarrow, "pr"]
\end{tikzcd}
\end{equation}
(The diagram \eqref{diagram-1} is commutative with no assumption, 
hence $\rho_{A,n}(F_{1}^{2}) \subset {\rm fil}_{{\rm HS}}^{2}H^{2d+1}$. 
Consider the diagram \eqref{admiring diagram} for $(s, r) = (1, 1)$. 
The lower square in \eqref{admiring diagram} commutes 
by chasing the proof of Lemma \ref{hs&cup} again for $(s, r) = (1, 1)$. 
We should show that $pr_{1,1} \circ \psi_{A} (\alpha) = pr \circ c_{1,1} (\alpha)$ in $E_{\infty}^{2, 2d-1}$, 
using the infiniteness of $k$.)
Hence $\rho_{A,n}(F_{1}^{3}) \subset {\rm fil}_{{\rm HS}}^{3}H^{2d+1} = 0$, that is, 
$F_{1}^{3} \subset \mathrm{Ker} (\rho_{A, n}^{1})$.

The cycle map $\rho_{A, n}^{s}$ is the zero-map for $s \geq 3$.
Yamazaki has shown in \cite[Theorem 4.3]{Y} that if $A$ has split multiplicative reduction, 
the condition $(*)_{0}$ holds. Therefore we have
$$
\mathrm{Ker} (\rho_{A}^{0} / n) \otimes \mathbb{Z}[1/2] 
= \frac{F_{0}^{3} + n\mathrm{CH}_{0}(A)}{n\mathrm{CH}_{0}(A)} \otimes \mathbb{Z}[1/2]
$$
for such $A$. 
This is  a result in \cite[Proposition 6.8]{G} and 
our proof of Corollary \ref{evaluation kernel} follows her argument. 
In \cite[Appendix]{Y} one finds a result of Spiess that if $A$ is the Jacobian variety of a smooth projective geometrically connected curve $C$ over $k$ with $C(k) \neq \emptyset$, the condition $(*)_{1}$ holds. Therefore we have
$$
\mathrm{Ker} (\rho_{A}^{1} / n) 
= \frac{F_{1}^{3} + n\mathrm{CH}^{d+1}(A, 1)}{n\mathrm{CH}^{d+1}(A, 1)}
$$
for such $A$ and odd $n$. 
When $s = 2$, one always has $\mathrm{Ker} (\rho_{A}^{2}/n) = (F_{2}^{3} + nF_{2}^{0}) / nF_{2}^{0}$.
\end{remark}

\section{Local field}																				\label{Local field}
In this section, we assume that $k$ is a finite extension of $\mathbb{Q}_{p}$. 
Let $X$ be a proper smooth integral scheme over $k$ of dimension $d$. 
Let $n > 0$ be any integer 
(which may be even).
Let $i, j$ be integers. We use the perfect pairing of finite abelian groups \cite[2.9.~Lemma]{Sa}:
\begin{equation}																					\label{P-T}
H^{i}(X, \mathbb{Z}/n(j)) \times H^{2d+2-i}(X, \mathbb{Z}/n(d+1-j)) \to \mathbb{Z}/n.			\tag{\ref{Local field}.0.1}
\end{equation}
In section \ref{B-M}, we review a result of Gazaki on the Brauer-Manin pairing, 
and in section \ref{reciprocity map}, we apply a similar argument for the reciprocity map. \\
\noindent \textit{Conventions.} 
For an abelian group $B$, we denote $\mathrm{Hom}(B, \mathbb{Q} / \mathbb{Z})$ by $B^{\vee}$ 
and by $B_{\text{div}}$ the maximal divisible subgroup of $B$. For any scheme $S$, let $\mathrm{Br}(S)$ be the cohomological Brauer group $H^{2}(S_{\text{\'{e}t}}, \mathbb{G}_{\mathrm{m}})$.
	\subsection{The Brauer-Manin pairing}													\label{B-M}
		We have the Brauer-Manin pairing 
$\mathrm{CH}_{0}(X) \times \mathrm{Br}(X) \to \mathrm{Br}(k) \cong \mathbb{Q} / \mathbb{Z}$,
where the second isomorphism is the invariant map defined in the local class field theory of $k$. 
We obtain a homomorphism $\Psi_{X} : \mathrm{CH}_{0}(X) \to \mathrm{Br}(X)^{\vee}$.
Since $\mathrm{Br}(X)$ is a torsion group and $\mathrm{Br}(X)[n]$ are finite groups, 
$\mathrm{Br}(X)^{\vee}$ is a profinite group.
We review a relation between the Brauer-Manin pairing and the cycle map.
There are the following commutative diagrams:

\begin{figure}[h]
\centering
\begin{minipage}{0.45\columnwidth}
\centering
	\begin{tikzcd}
	\mathrm{CH}_{0}(X) / n \arrow[r, "\rho_{X} / n"]\arrow[d, "\Psi_{X} / n"']
	&H^{2d}(X, \mathbb{Z}/n(d)) \arrow{d}[sloped, above]{\sim} \\
	(\mathrm{Br}(X)[n])^{\vee} \arrow[r, hook, "\lambda_{X}^{\vee}"']
	&H^{2}(X, \mathbb{Z}/n(1))^{\vee}
	\end{tikzcd}
\end{minipage}
\begin{minipage}{0.45\columnwidth}
\centering
	\begin{tikzcd}
	\mathrm{CH}_{0}(X) \arrow[r, "\rho_{X}"]\arrow[d, "\Psi_{X}"']
	&H^{2d}(X, \Hat{\mathbb{Z}}(d)) \arrow{d}[sloped, above]{\sim} \\
	\mathrm{Br}(X)^{\vee} \arrow[r, hook, "\lambda_{X}^{\vee}"']
	&H^{2}(X, \mathbb{Q} / \mathbb{Z}(1))^{\vee}
	\end{tikzcd}
\end{minipage}
\end{figure}
\noindent where the right vertical isomorphisms are induced by \eqref{P-T}, 
and $\lambda_{X}$ are induced by the Kummer exact sequence
$0 \to \mu_{n} \to \mathbb{G}_{\mathrm{m}} \xrightarrow{n} \mathbb{G}_{\mathrm{m}} \to 0$
and that $\mathrm{Br}(X)$ is torsion.
In particular, we have $\mathrm{Ker} \, \Psi_{X} = \mathrm{Ker} \, \rho_{X} = \bigcap_{n} \mathrm{Ker} (\rho_{X,n})$.

By Remark \ref{remark 2}, 
we obtain 
$\bigcap_{n}(F_{0}^{3} + nF_{0}^{0}) \otimes \mathbb{Z}[1/2] \subset 
\mathrm{Ker} \, \Psi_{A} \otimes \mathbb{Z}[1/2] \subset 
\bigcap_{n}(F_{0}^{2} + nF_{0}^{0}) \otimes \mathbb{Z}[1/2]$ 
for an abelian variety $A$ over $k$.

\begin{thm} {\rm (\cite[Corollary 6.3, Theorem 6.9]{G})}										\label{div&Psi}
\begin{enumerate}
\item
If $A$ has split semi-ordinary reduction,
then $F_{0}^{\nu} / F_{0}^{\nu+1}$ is divisible group if $\nu \geq 3$
and $F_{0}^{2} / F_{0}^{3} \otimes \mathbb{Z}[1/2]$ is the direct sum of a finite group and a divisible group.
\item 
Let $K_{0} = \mathrm{Ker} \, \Psi_{A}$.
Then $F_{0}^{3} \otimes \mathbb{Z}[1/2] \subset K_{0} \otimes \mathbb{Z}[1/2] \subset F_{0}^{2} \otimes \mathbb{Z}[1/2]$.
Moreover if 
$A$ has split multiplicative reduction, then 
$(K_{0} / F_{0}^{\nu}) \otimes \mathbb{Z}[1/2] 
= (F_{0}^{2} / F_{0}^{\nu})_{\rm{div}} \otimes \mathbb{Z}[1/2]$
for $\nu \geq 3$.
\end{enumerate}
\end{thm}

See \cite[Section 4.2]{G} for a discussion on the vanishing of $F_{0}^{\nu} \otimes \mathbb{Q}$ for $\nu >> 0$.
	\subsection{The reciprocity map}														\label{reciprocity map}
		We refer the reader to \cite{Sa2} for the definition of the group $SK_{1}(X)$, 
the norm map $\mathrm{Nm} : SK_{1}(X) \to k^{*}$, 
and the reciprocity map $\mathrm{rec}_{X} : SK_{1}(X) \to \pi_{1}^{\rm{ab}}(X)$. 
We review a relation between the reciprocity map and the cycle map.
There are the following commutative diagrams:

\begin{figure}[h]
\centering
\begin{minipage}{0.45\columnwidth}
\centering
	\begin{tikzcd}
	\mathrm{CH}^{d+1}(X, 1) \arrow[r, "\rho_{X,n}"]\arrow{d}[sloped, below]{\sim}
	&H^{2d+1}(X, \mathbb{Z}/n(d+1)) \arrow{d}[sloped, above]{\sim} \\
	SK_{1}(X) \arrow[r, "\mathrm{rec}_{X,n}"']
	&\pi_{1}^{\rm{ab}}(X) / n
	\end{tikzcd}
\end{minipage}
\begin{minipage}{0.45\columnwidth}
\centering
	\begin{tikzcd}
	\mathrm{CH}^{d+1}(X, 1) \arrow[r, "\rho_{X}"]\arrow{d}[sloped, below]{\sim}
	&H^{2d+1}(X, \Hat{\mathbb{Z}}(d+1)) \arrow{d}[sloped, above]{\sim} \\
	SK_{1}(X) \arrow[r, "\mathrm{rec}_{X}"']
	&\pi_{1}^{\rm{ab}}(X)
	\end{tikzcd}
\end{minipage}
\end{figure}
\noindent where the right vertical isomorphisms are induced by \eqref{P-T}.
In particular, we have 
$\mathrm{Ker} (\mathrm{rec}_{X}) \simeq \mathrm{Ker} \, \rho_{X} = \bigcap_{n} \mathrm{Ker} (\rho_{X,n})$.

By Remark \ref{remark 2}, 
we obtain 
$\bigcap_{n}(F_{1}^{3} + nF_{1}^{0}) \subset \mathrm{Ker}(\rho_{A}) \subset \bigcap_{n}(F_{1}^{2} + nF_{1}^{0})$ 
for an abelian variety $A$ over $k$.

\begin{thm}																						\label{div&kernel}
\begin{enumerate}
\item
Assume that $A$ has potentially good reduction or split semi-abelian reduction.
Then $F_{s}^{\nu} / F_{s}^{\nu+1}$ is divisible group if $s > 0$ and $\nu \geq 3$.
\item
Let $K_{1} = \mathrm{Ker} (\mathrm{CH}^{d+1}(A, 1) \to \pi_{1}^{{\rm ab}}(A))$.
Then $F_{1}^{3} \subset K_{1} \subset F_{1}^{2}$.
If $A$ is the Jacobian $\mathrm{Jac} (C)$ 
of a smooth proper geometrically connected curve $C$ over $k$ with $C(k) \neq \emptyset$,
then $F_{1}^{2} / F_{1}^{3}$ is the direct sum of a finite group and a divisible group
and $K_{1} / F_{1}^{3} \otimes \mathbb{Z}[1/2] 
= ( F_{1}^{2} / F_{1}^{3} )_{{\rm div}} \otimes \mathbb{Z}[1/2] 
\cong \mathrm{Ker}(\mathrm{rec}_{C}) \otimes \mathbb{Z}[1/2]$.
\end{enumerate}
\end{thm}

\begin{proof}
(i) By \cite[Lemma 2.4, Proposition 3.1]{Y2},
the Mackey product 
$A^{\stackrel{\mathrm{M}}{\otimes} r} \stackrel{\mathrm{M}}{\otimes} 
\mathbb{G}_{\mathrm{m}}^{\stackrel{\mathrm{M}}{\otimes} s} (k)$
is divisible if $r+s \geq 3$, $s > 0$.
Then $S_{r}(k; A, \mathcal{K}_{s}^{\mathrm{M}})$ is also divisible,
hence we have $r!S_{r}(k; A, \mathcal{K}_{s}^{\mathrm{M}}) = S_{r}(k; A, \mathcal{K}_{s}^{\mathrm{M}})$.
Therefore \eqref{injection} is also surjective by Proposition \ref{Psi_{r,s}}.

\noindent (ii) 
We have $F_{1}^{3} \subset \bigcap_{n} \mathrm{Ker} (\rho_{A,n}) = K_{1}$.
Let $\pi : A \to \mathrm{Spec} \, k$ be the structure morphism.
We define $\pi_{1}^{\rm{geo}}(X) := \mathrm{Ker} (\pi_{*} : \pi_{1}^{\rm{ab}}(X) \twoheadrightarrow \pi_{1}^{\rm{ab}}(k))$ 
and $V(X) := \mathrm{Ker} (\mathrm{Nm} : SK_{1}(X) \to k^{*})$.
We have the following commutative diagram
\begin{center}
\begin{tikzcd}
F_{1}^{2} \arrow[r, hook]\arrow{d}[sloped, below]{\sim}
&\mathrm{CH}^{d+1}(A, 1) \arrow[r, twoheadrightarrow, "\Phi_{0,1}"]\arrow{d}[sloped, below]{\sim}
&K_{1}^{\mathrm{M}}(k) \arrow[d, equal] \\
V(A) \arrow[r, hook]\arrow[d]
&SK_{1}(A) \arrow[r, "\rm{Nm}"]\arrow[d, "\mathrm{rec}_{A}"']
&k^{*} \arrow[d, hook, "\mathrm{rec}_{k}"] \\
\pi_{1}^{\rm{geo}}(A) \arrow[r, hook]
&\pi_{1}^{\rm{ab}}(A) \arrow[r, twoheadrightarrow, "\pi_{*}"']
&\pi_{1}^{\rm{ab}}(k).
\end{tikzcd}
\end{center}
Here the injectivity of $\mathrm{rec}_{k}$ follows from the local class field theory of $k$, 
which yields 
$K_{1} \subset F_{1}^{2}$ by using Theorem \ref{canonical isomorphisms}.
Now we assume $A = \mathrm{Jac} (C)$ and denote by $J$.
We have $F_{1}^{2} / F_{1}^{3} \cong K(k; J, \mathbb{G}_{\mathrm{m}}) \cong V(C)$,
where the second isomorphism is defined in \cite[Theorem 2.1]{S}.
It is concluded that 
$F_{1}^{2} / F_{1}^{3}$ is the direct sum of a finite group and a divisible group
by Corollary 5.2 in \cite{Sa2}
and that $( F_{1}^{2} / F_{1}^{3} )_{{\rm div}} \cong \mathrm{Ker}(\mathrm{rec}_{C})$
by Theorem 5.1 in \cite{Sa2}.
Take the subgroup $F_{1}^{3} \subset D \subset F_{1}^{2}$ corresponding to $( F_{1}^{2} / F_{1}^{3} )_{{\rm div}}$.
It is enough to show $D = K_{1}$.
The composition
$D / F_{1}^{3} \to \pi_{1}^{\mathrm{ab}} (J) \twoheadrightarrow \pi_{1}^{\mathrm{ab}} (J) / n$
is zero map for any integer $n > 0$
since $D / F_{1}^{3}$ is divisible
and $\pi_{1}^{\mathrm{ab}} (J) / n$ is finite.
Therefore $D \subset K_{1}$.
We know $(F_{1}^{2} / F_{1}^{3}) / ( F_{1}^{2} / F_{1}^{3} )_{{\rm div}}$ is a finite group.
Let $n_{1}$ be its order, then $n_{1}(F_{1}^{2} / F_{1}^{3}) = ( F_{1}^{2} / F_{1}^{3} )_{{\rm div}}$. 
It holds that $F_{1}^{2} / (n_{1}F_{1}^{2} + F_{1}^{3}) = F_{1}^{2} / D$.
We have the following commutative diagrams
\begin{center}
\begin{tikzcd}
F_{1}^{2} / (n_{1}F_{1}^{2} + F_{1}^{3}) \arrow[r, "\sim"]\arrow[d, "\rho_{J} / n_{1}"'] 
&K(k; J, \mathbb{G}_{\mathrm{m}}) / n_{1} \arrow[d, hook, "s_{n_{1}}"] \\
\mathrm{fil}_{\mathrm{HS}}^{2}H^{2d+1}(J, \mathbb{Z} / n_{1}(d+1)) \arrow[r, dotted, "\sim"]
&H^{2}(k, J[n_{1}] \otimes \mu_{n_{1}})
\end{tikzcd}
\end{center}
\begin{center}
\begin{tikzcd}
0 \arrow[r] &D \arrow[r] &F_{1}^{2} \arrow[r]\arrow[d, "\rho_{J}"'] 
&F_{1}^{2} / (n_{1}F_{1}^{2} + F_{1}^{3}) \arrow[r]\arrow[d, "\rho_{J} / n_{1}"] &0 \\
&&H^{2d+1}(J, \Hat{\mathbb{Z}}(d+1)) \arrow[r] &H^{2d+1}(J, \mathbb{Z} / n_{1}(d+1)) \arrow[r] &0
\end{tikzcd}
\end{center}
where in the first diagram, the bottom dotted isomorphism exists only after $\otimes \mathbb{Z}[1/2]$
and the injectivity of $s_{n_{1}}$ has been shown in \cite[Appendix]{Y}.
The kernel of $\rho_{J}$ in the second diagram is equal to $K_{1}$ since $K_{1} \subset F_{1}^{2}$.
Hence we deduce $K_{1} \otimes \mathbb{Z}[1/2] \subset D \otimes \mathbb{Z}[1/2]$ 
from the injectivity of $\rho_{J} / n_{1} \otimes \mathbb{Z}[1/2]$ in the second diagram.
\end{proof}

\noindent \textit{Acknowledgement.}
I thank to my graduate supervisor Prof. Takao Yamazaki 
for his unstinted support and encouragement.
I was suggested these problems from him and was lectured various things.
He read this paper over again and provided many valuable comments for improvement.
In particular, the proof of Proposition \ref{prop phi} has completed by his suggestion 
that using $K_{r}(k; A, \mathcal{K}_{s}^{\mathrm{M}})$ 
instead of $K_{r,s}(k; A, \mathbb{G}_{\mathrm{m}})$ used in the initial version.
Section \ref{subsection H-S} and \ref{cycle-somekawa} were proven by him for any (perfect) field.

\end{document}